\documentclass[12pt]{article}
\usepackage{amssymb,amsfonts,url,amsmath}
\usepackage{hyperref}

\newtheorem{theorem}{Theorem}[section]
\newtheorem{cor}[theorem]{Corollary}
\newtheorem{prop}[theorem]{Proposition}
\newtheorem{lemma}[theorem]{Lemma}
\newcommand{\qed}{\qquad$\Box$}
\newtheorem{unnumber}{}

\newtheorem{prepf}[unnumber]{Proof}
\newenvironment{proof}{\prepf\rm}{\endprepf}

\newcommand{\Sym}{\mathop{\mathrm{Sym}}}

\newcommand{\JI}{\mathop{\mathrm{JI}}}
\newcommand{\im}{\mathop{\mathrm{im}}}
\newcommand{\id}{\mathop{\mathrm{id}}}
\newcommand{\blobb}{\circle*{1}}
\newcommand{\blob}{\circle{1}}
\newcommand{\bloz}{\circle{2}}

\begin{document}

\title{Permutation groups, partition lattices and block structures}
\author{Marina Anagnostopoulou-Merkouri,\\
\small{School of Mathematics, University of Bristol}\\
R. A. Bailey and Peter J. Cameron\\
\small{School of Mathematics and Statistics, University of St Andrews}}
\date{}
\maketitle

\begin{abstract}
Let $G$ be a finite transitive permutation group on $\Omega$. The $G$-invariant
partitions form a sublattice of the lattice of all partitions of $\Omega$,
having the further property that all its elements are uniform (that is, have
all parts of the same size). If, in addition, all the equivalence relations
defining the partitions commute, then the relations form an
\emph{orthogonal block structure}, a concept from statistics; in this case
the lattice is modular. If it is distributive, then we have a \emph{poset
block structure}, whose automorphism group is a \emph{generalised wreath
product}. We examine permutation groups with these properties, which we
call the \emph{OB property} and \emph{PB property} respectively, and in
particular investigate when direct and wreath products of groups with these
properties also have these properties.

A famous theorem on permutation groups asserts that a transitive
imprimitive group $G$
is embeddable in the wreath product of two factors obtained from the group
(the group induced on a block by its setwise stabiliser, and the group induced
on the set of blocks by~$G$). We extend this theorem to groups with the PB
property, embedding them into generalised wreath products. We show that
the map from posets to generalised wreath products preserves intersections
and inclusions.

We have included background and historical material on these concepts.


\end{abstract}

\section{Introduction}
\label{s:intro}

Let $G$ be a finite transitive permutation group on a finite set $\Omega$. Then the
$G$-invariant partitions of $\Omega$ form a sublattice of the lattice of all
partitions of $\Omega$ (ordered by refinement). The $G$-invariant partitions
have the additional property that they are \emph{uniform} (all parts have
the same size).

In this paper all permutation groups will be finite and we are primarily interested in the class of permutation groups
for which the equivalence relations corresponding to the $G$-invariant
partitions commute pairwise. (We will see in Section~\ref{s:comput} that,
at least among transitive groups of small degree, the vast majority do satisfy
this condition; for example, $1886$ of the $1954$ transitive groups of
degree~$16$ do so.) Then  the lattice of partitions which they form is called
an \emph{orthogonal block structure}, for short an OBS.
This property can also be defined by saying that
the subgroups containing a point stabiliser $G_\alpha$ commute pairwise.
This implies that the lattice satisfies the \emph{modular law}. It
turns out that this property of a partition lattice was introduced, in the
context of statistical design, by several different statisticians: see
Section~\ref{sec:hist}.

An orthogonal block structure gives rise, by an inclusion-exclusion argument,
to an association scheme on $\Omega$; we also explain this and its relevance
to the study of permutation groups.

A more restrictive property requires that the lattice satisfies the
\emph{distributive law}. These structures are known, in the statistical
context, as \emph{poset block structures}.
These are explained in Section~\ref{sec:hist}.  The simplest non-trivial cases
are
(i)~a single non-trivial uniform partition 
and (ii)~the rows and columns of  a rectangle.
These correspond to the imprimitive wreath product and the transitive direct
product of two permutation groups.

This is related to an earlier permutation group construction, the so-called
\emph{generalised wreath product}. This takes as input data a partially
ordered set $M$ having a transitive permutation group associated with each of
its elements, and produces a product which generalises both direct and wreath
product (the cases where the poset is a $2$-element antichain or $2$-element
chain respectively). The \emph{Krasner--Kaloujnine theorem},
a well-known theorem in permutation group theory, describes the embedding
of a transitive but imprimitive permutation group in a wreath prduct; we
generalise this to embed a group whose invariant partitions form a poset
block structure into a generalised wreath product over the poset.

We say that a transitive group $G$ has the \emph{OB property} (respectively
\emph{PB property}) if the $G$-invariant partitions form an orthogonal block
structure (respectively a poset block structure). We investigate some
properties of
these groups, including their behaviour under direct and wreath products, and
characterise the regular groups with the OB property (using a theorem of
Iwasawa).

A summary of the paper follows. In Section~\ref{sec:partlat}, we give precise
definitions of orthogonal and poset block structures and the generalised
wreath product of a family of permutation groups indexed by a poset.
Section~\ref{sec:hist} describes the history of these block structures
in experimental design in statistics. Section~\ref{s:pg} contains our
main results on permutation groups. We give somewhat informal descriptions
here, since precise statements depend on the notions of generalised wreath
product and the OB and PB properties.
\begin{enumerate}
\item We show that a generalised wreath product of primitive permutation groups
is pre-primitive and has the OB property, and we give a necessary and
sufficient condition for it to have the PB property: the obstruction is the
existence of incomparable elements in the poset whose associated groups are
cyclic of the same prime order (Theorem~\ref{thm-gwpPP}).
\item We show that a transitive group $G$ which acts on a poset block structure
(in particular, a transitive group with the PB property) can be embedded
in a generalised wreath product, where the factors in the product can be
defined in terms of the action of $G$ (Theorem~\ref{t:embed}).
\item The map from posets on the index set to generalised wreath products of
families of groups preserves intersections and inclusions, where for a poset
these refer to the set of ordered pairs comprising the relation.  In 
particular, a generalised wreath product is the intersection of the iterated
wreath products over all linear extensions of the poset
(Theorem~\ref{t:poset2gwp}, Corollary~\ref{c:linext}).
\end{enumerate}
We also examine the behaviour of OB and PB under direct and wreath product.

The final Section~\ref{s:misc} describes some computational issues and gives
some open problems.

Since the paper crosses boundaries between permutation groups, lattice theory,
and statistical design, we have given some introductory material on these
topics (Section~\ref{sec:partlat}), as well as an account of the somewhat tangled history
of their occurrence in statistics (Section~\ref{sec:hist}).

\vspace{3mm}

\noindent \textbf{Acknowledgements.} We are grateful to Michael Kinyon for drawing our attention to the paper~\cite{yan}. We thank the anonymous referees for a careful reading of the paper and all the useful comments and suggestions that have greatly improved the paper. The first author thanks the School of Mathematics and Statistics of the University of St Andrews for supporting part of this work with an undergraduate StARIS scholarship. The rest of the work was undertaken during the first author's doctoral studies at the University of Bristol and she thanks the Heilbronn Institute for Mathematical Research for the financial support.

\section{Lattices of Partitions}
\label{sec:partlat}
\subsection{Partitions}

Let $\Omega$ be a finite set. The set of all partitions of $\Omega$ is
partially ordered by refinement: $\Pi_1\preccurlyeq\Pi_2$ if each part of
$\Pi_1$ is contained in a part of $\Pi_2$. With this order, the partitions
form a lattice (a partially ordered set in which any two elements have a
greatest lower bound or \textit{meet}, and a least upper bound or \textit{join}): the meet
(also called \textit{infimum})
$\Pi_1\wedge\Pi_2$ is the partition whose parts are
all non-empty intersections of parts of $\Pi_1$ and $\Pi_2$, and the join
(also called \textit{supremum})
$\Pi_1\vee\Pi_2$ is the partition in which the part containing $\alpha$
consists of all points of $\Omega$ that can be reached from $\alpha$ by
moving alternately within a part of $\Pi_1$ and within a part of $\Pi_2$.

Partitions can be considered also as equivalence relations. The composition
$R_1\circ R_2$ of two relations $R_1$ and $R_2$ is the relation in which
$\alpha$ and $\beta$ are related if and only if there exists $\gamma$ with
$(\alpha,\gamma)\in R_1$ and $(\gamma,\beta)\in R_2$.

In view of the natural correspondence between partitions and equivalence
relations, we abuse notation by talking about the join $R_1\vee R_2$ of
two equivalence relations, or the composition $\Pi_1\circ\Pi_2$ of two
partitions.

\begin{prop}
$R_1\circ R_2=R_1\vee R_2$ if and only if $R_1\circ R_2=R_2\circ R_1$.
\end{prop}

\begin{proof}
Clearly $R_1\circ R_2\subseteq R_1\vee R_2$.

Suppose that $R_1\circ R_2=R_2\circ R_1$. If $\alpha$ and $\beta$ lie in the
same part of $R_1\vee R_2$, then there is a path joining them, whose edges
lie alternately in the same part of $R_1$ and of $R_2$. But any three
consecutive steps $(\alpha_1,\alpha_2,\alpha_3,\alpha_4)$ with
$(\alpha_1,\alpha_2),(\alpha_3,\alpha_4)\in R_1$ and $(\alpha_2,\alpha_3)\in R_2$ can be
shortened to two steps: for there exists $\beta'$ with $(\alpha_1,\beta')\in R_2$
and $(\beta',\alpha_3)\in R_1$; then $(\beta',\alpha_4)\in R_1$ by transitivity.
So $R_1\vee R_2=R_1\circ R_2$.

Conversely, suppose that $R_1\circ R_2=R_1\vee R_2$. Then $R_1\circ R_2$ is
symmetric, so it is equal to $R_2\circ R_1$.\qed
\end{proof}

This result was first proved in~\cite{ddj}.

\subsection{Lattices}

A finite lattice is conveniently represented by its \emph{Hasse diagram}:
this is the plane diagram with a dot for each lattice element; if
$a\prec b$ then $b$ is higher than $a$ in the plane; and if $b$ covers $a$
(that is, $a\prec b$ but there is no element~$c$ with $a\prec c\prec b$), then
an edge joins $a$ to $b$.

In a lattice, the \emph{modular law} states that
\[a\preccurlyeq c \hbox{ implies }a\vee(b\wedge c)=(a\vee b)\wedge c.\]
A lattice $L$ is \emph{modular} if this holds for all $a,b,c\in L$.

\begin{prop}\label{prop:OBS-modular}
  In a lattice of partitions, if every pair of 
  partitions commute, then the lattice is modular.
\label{p:modular}
\end{prop}

\begin{proof}
We are required to prove that $\Phi\preccurlyeq \Psi$ implies
$\Phi\vee(\Xi\wedge \Psi)=(\Phi\vee \Xi)\wedge \Psi$.
In Figure~\ref{f:mod}, the dots represent points
in~$\Omega$.  Each edge is labelled by a partition of $\Omega$. If an edge labelled~$\Phi$
joins points $\alpha$ and $\beta$, this means that $\alpha$ and $\beta$ are in the same
part of~$\Phi$;
and similarly for $\Xi$ and $\Psi$.

\begin{figure}[htbp]
\begin{center}
\setlength{\unitlength}{1mm}
\begin{picture}(80,20)
\multiput(0,10)(10,0){3}{\circle*{2}}
\multiput(50,10)(20,0){2}{\circle*{2}}
\put(60,5){\circle*{2}}
\put(0,10){\line(1,0){10}}\put(4,11){$\Phi$}
\put(10,11){\line(1,0){10}}\put(14,12){$\Xi$}
\put(10,9){\line(1,0){10}}\put(14,6){$\Psi$}
\put(50,10){\line(1,0){20}}\put(59,11){$\Psi$}
\put(50,10){\line(2,-1){10}}\put(54,4){$\Phi$}
\put(60,5){\line(2,1){10}}\put(64,4){$\Xi$}
\put(-2,10){\makebox(0,0)[r]{$\alpha$}}
\put(22,10){\makebox(0,0)[l]{$\gamma$}}
\put(9,9){\makebox(0,0)[rt]{$\beta$}}
\put(48,10){\makebox(0,0)[r]{$\zeta$}}
\put(72,10){\makebox(0,0)[l]{$\theta$}}
\put(60,3){\makebox(0,0)[t]{$\eta$}}
\end{picture}
\end{center}
\caption{\label{f:mod}The modular law for commuting partitions}
\end{figure}

Since $\Phi\preccurlyeq\Psi$, any $\Phi$-$\Psi$ path can be replaced by a single
$\Psi$ edge. So, considering the paths from $\alpha$ to $\gamma$ in the diagram on
the left shows that
$\Phi\vee(\Xi\wedge \Psi)\preccurlyeq(\Phi\vee \Xi)\wedge \Psi$.
Also, on the right, the $\Psi$-$\Phi$ path from $\theta$ to $\eta$
implies that there is a $\Psi$ edge between them. Thus there is a $\Xi \wedge \Psi$ path from
$\theta$ to $\eta$, and hence a $(\Xi \wedge \Psi)\vee\Phi$ path from $\theta$ to $\zeta$:
this gives the reverse inequality.
\qed
\end{proof}

Proposition~\ref{prop:OBS-modular} is Theorem 9.11 in~\cite{bailey:as} and Proposition 8 in~\cite{Rota1996}.

\vspace{3mm}

A lattice is \emph{distributive} if it satisfies the conditions
\begin{eqnarray*}
(a\vee b)\wedge c &=& (a\wedge c)\vee(b\wedge c),\\
(a\wedge b)\vee c &=& (a\vee c)\wedge(b\vee c),
\end{eqnarray*}
for all $a,b,c$.

\begin{prop}\label{prop:dist-laws}
\begin{enumerate}
\item Each of the two distributive laws implies the other.
\item A distributive lattice is modular.
\end{enumerate}
\label{p:dist-l}
\end{prop}

\begin{proof}
(a) Suppose that the first law above holds. Then
\begin{eqnarray*}
(a\vee c)\wedge(b\vee c) &=& ((a\vee c)\wedge b)\vee((a\vee c)\wedge c)\\
&=&(a\wedge b)\vee(c\wedge b)\vee c\\
&=&(a\wedge b)\vee c.
\end{eqnarray*}
The proof of the other implication is similar.

\medskip

(b) Suppose that $L$ is
distributive and let $a,b,c\in L$ with $a\preccurlyeq c$. Then
\[a\vee(b\wedge c) = (a\vee b)\wedge(a\vee c)=(a\vee b)\wedge c,\]
since $a\preccurlyeq c$ implies $a\vee c=c$.\qed
\end{proof}

Proposition~\ref{prop:dist-laws} is a standard result in lattice theory and appears in~\cite{dp} as Lemmas 4.2 and 4.3.

\medskip

The fundamental theorem on distributive lattices states that
every finite distributive lattice is isomorphic to a sublattice of the
\emph{Boolean lattice} of all subsets of a finite set. More precisely, a
\emph{down-set} in a partially ordered set $(M,{\sqsubseteq})$ is a subset $D$
of $M$ with the property that, if $m\in D$ and $m'\sqsubseteq m$, then
$m'\in D$. The down-sets form a lattice under the operations of intersection
and union.

\begin{theorem}
A finite distributive lattice $L$ is isomorphic to the lattice of down-sets in
a partially ordered set $M$. We can take $M$ to be the set of
join-indecomposable elements of $L$ (elements $m$ for which $m=m_1\vee m_2$
implies $m=m_1$ or $m=m_2$).
\label{t:ftdl}
\end{theorem}

A proof of this theorem is in \cite[p.~192]{CTTA}. We sometimes abbreviate
``join-indecomposable'' to JI.

In particular, if $M$ is an antichain (a poset in which any two elements are
incomparable), then every subset is a down-set, and
the corresponding lattice is the \emph{Boolean lattice} on $M$.

There are well-known characterisations of these classes of lattices. The
Hasse diagrams of $P_5$ and $N_3$ are shown in Figure~\ref{f:p5n3}.

\begin{theorem}\label{thm:diamond-pentagon}
\begin{enumerate}
\item A lattice is modular if and only if it does not contain $P_5$ as a
sublattice.
\item A lattice is distributive if and only if it does not contain $P_5$ or
$N_3$ as a sublattice.
\end{enumerate}
\end{theorem}

\begin{figure}[htbp]
\begin{center}
\setlength{\unitlength}{1mm}
\begin{picture}(54,24)
\multiput(10,0)(0,24){2}{\circle*{1}}
\multiput(2,8)(0,8){2}{\circle*{1}}
\put(18,12){\circle*{1}}
\put(10,0){\line(-1,1){8}}
\put(10,24){\line(-1,-1){8}}
\put(2,8){\line(0,1){8}}
\put(10,0){\line(2,3){8}}
\put(10,24){\line(2,-3){8}}
\multiput(30,12)(12,0){3}{\circle*{1}}
\multiput(42,0)(0,24){2}{\circle*{1}}
\multiput(42,0)(12,12){2}{\line(-1,1){12}}
\multiput(42,0)(-12,12){2}{\line(1,1){12}}
\put(42,0){\line(0,1){24}}
\end{picture}
\end{center}
\caption{\label{f:p5n3}The lattices $P_5$ (left) and $N_3$ (right)}
\end{figure}

The proof of this theorem can be found in \cite[p.~134]{dp}.

\subsection{Orthogonal block structures}
\label{sec:OBS}

The next definition comes from experimental design in statistics: see the discussion
in Section~\ref{sec:hist}.
Our treatment follows~\cite{bailey:as}.

An \emph{orthogonal block structure} $(\Omega,\mathcal{B})$ consists of
a collection $\mathcal{B}$ of
partitions of a single set~$\Omega$ satisfying the conditions
\begin{enumerate}
\item $\mathcal{B}$ is a sublattice of the partition lattice (that is,
closed under meet and join);
\item $\mathcal{B}$ contains the two extreme partitions (the equality partition $E$ whose
parts are singletons, and the universal partition $U$ with just one part);
\item every partition in $\mathcal{B}$ is uniform (that is, has all parts of
the same size);
\item any two partitions in $\mathcal{B}$ commute.
\end{enumerate}
The set $\mathcal{B} = \{E,U\}$ is an orthogonal block structure, which we call
\textit{trivial}.


We remark that the definition in \cite[Chapter 6]{bailey:as} has a more
complicated condition in place of our condition (d). With any partition $\Pi$
is associated a subspace $V_\Pi$ of the vector space $\mathbb{R}^\Omega$
consisting of functions which are constant on the parts of $\Pi$, and the
operator $P_\Pi$ of orthogonal projection of $\mathbb{R}^\Omega$ onto $V_\Pi$;
two partitions $\Pi_1$ and $\Pi_2$ are said to be orthogonal if $P_{\Pi_1}$
and $P_{\Pi_2}$ commute. The remark at the top of page 153 of \cite{bailey:as}
notes that, in the presence of conditions (a)--(c), this is equivalent to our simpler condition
(d).

An \emph{association scheme} on $\Omega$ is a partition of $\Omega^2$ into
symmetric relations $S_0,S_1,\ldots,S_r$ having the properties that $S_0$
is the relation of equality and that the span over $\mathbb{R}$ of the
zero-one relation matrices is an algebra. (Combinatorially this means that,
given $i,j,k\in\{0,\ldots,r\}$ and $\alpha$, $\beta\in\Omega$ with $(\alpha,\beta)\in S_k$,
the number $p_{ij}^k$ of elements $\gamma\in\Omega$ such that $(\alpha,\gamma)\in S_i$
and $(\gamma,\beta)\in S_j$ is independent of the choice of $(\alpha,\beta)\in S_k$,
depending only on $i,j,k$.)

An orthogonal block structure gives rise to an association scheme as follows.
Let $R_0$, $R_1$ , \ldots, $R_t$ be equivalence relations forming an OBS. For each~$i$, let
\[S_i=R_i\setminus\bigcup_{j:R_j\subset R_i}R_j.\]
Then the non-empty relations $S_i$ are symmetric and partition $\Omega^2$;
after removing the empty ones and re-numbering, we obtain an association
scheme.

The non-equality relations in an association scheme are often thought of as
graphs. We remark that, while in the association scheme associated with a
primitive permutation group, all these graphs are connected, the association
scheme associated with an orthogonal block structure is very different: all
the graphs, except possibly the one associated with the universal relation $U$,
are disconnected.

Note that, if two OBSs are isomorphic, then the association schemes obtained
in this way are also isomorphic. The converse, however, is false, as the
following example shows.

\paragraph{Example} Recall that two Latin squares $A = (a_{ij})$ and $B = (b_{ij})$ over alphabets $\mathcal{A}$ and $\mathcal{B}$ are \emph{orthogonal} if for every pair $(a, \beta)\in \mathcal{A}\times \mathcal{B}$ there exists a unique pair $(i, j)$ such that $a_{ij} = a$ and $b_{ij} = \beta$. Note that it is a convention going back to Euler that the two alphabets are the Latin and Greek letters respectively when dealing with two orthogonal Latin squares.

Take a complete set of $q-1$ mutually orthogonal Latin
squares of order $q$. Take $\Omega$ to be the set of cells of the square and let $\mathcal{B}$ be the set containing the partitions $E$ and $U$, and the partitions into rows, columns,
and letters of each of the squares. We claim that $\mathcal{B}$ is an orthogonal block structure.

Conditions (a) and (c) are straightforward to verify and (b) holds by construction, so we will only show that the partitions commute. Let $\Pi_1, \Pi_2\in \mathcal{B}$. If $\Pi_1$ and $\Pi_2$ are the row and the column partitions, then they commute because any cell can be reached by first moving along a row and then a column or vice versa. 

If $\Pi_1$ is the row partition and $\Pi_2$ is a letter partition, then we note that each letter appears in every row, and so if $x_1$ and $x_2$ are two cells in rows $r_1$ and $r_2$ respectively containing the letters $a$ and $b$ in the Latin square corresponding to $\Pi_2$, then we can reach $x_2$ from $x_1$ either by first moving to the cell in $r_1$ containing the letter $b$ and then to $x_2$, or by first moving to the cell containing the letter $a$ in $r_2$ and then moving along $r_2$ to $x_2$. Therefore $\Pi_1$ and $\Pi_2$ commute. An entirely similar argument shows that $\Pi_1$ and $\Pi_2$ commute if $\Pi_1$ is the column partition.

Finally, suppose that $\Pi_1$ and $\Pi_2$ are both letter partitions and let $L_1$ and $L_2$ be the Latin squares corresponding to $\Pi_1$ and $\Pi_2$ respectively. Moreover, let $x_1$ and $x_2$ be cells that contain the letters $a$ and $b$ in $L_1$ and the letters $\alpha$ and $\beta$ in $L_2$. Orthogonality ensures that there is exactly one cell $x_3$ which contains the letter $a$ in the $L_1$ and the letter $\beta$ in $L_2$, so we can move from $x_1$ to $x_2$ via $x_3$. Similarly, there exists a unique cell $x_4$ that contains the letter $\alpha$ in  $L_2$ and the letter $b$ in $L_1$, so we can move from $x_1$ to $x_2$ via $x_4$, and so $\Pi_1\circ \Pi_2 = \Pi_2\circ \Pi_1$.

Since every pair of cells are either in the same row or column or carry the
same letter in one of the squares, constructing the association scheme obtains the empty relation from the universal partition $U$. So the association scheme has $q+1$
classes apart from the diagonal.

On the other hand, if we omit one of the Latin squares from the set, then the
remaining ones give an OBS with $q$ partitions apart from $E$ and $U$;
the last partition is recovered by deleting the pairs in all these from the
relation~$U$. So the association schemes are the same.

In particular, for $q=2$, we obtain two orthogonal block structures, one of
which is distributive and the other not, which give the same association
scheme.

\medskip

A similar inclusion-exclusion on subspaces of $\mathbb{R}^\Omega$ finds the
orthogonal decomposition of
$\mathbb{R}^\Omega$ into common eigenspaces for the matrices in the scheme.

Let us now record some remarks on association schemes.
\begin{itemize}
\item The product of two relation matrices is a linear combination of the
relation matrices, hence symmetric; thus any two relation matrices commute,
and the algebra associated with the association scheme (called its
\emph{Bose--Mesner algebra}) is commutative.
\item There is a more general notion, that of a \emph{homogeneous coherent
configuration}, defined as for association schemes but with the condition that
every relation is symmetric replaced by the weaker condition that the converse
of any relation in the configuration is another relation in the configuration.
Some authors (including Hanaki and Miyamoto~\cite{hm}) extend the usage of
the term ``association scheme'' to this more general situation; but we will
not do so.
\end{itemize}

\medskip

It was pointed out to us by the anonymous referee that partition lattices whose equivalence relations commute were independently introduced by G.-C. Rota and his students and collaborators, who called them \emph{linear lattices}. Details can be found in~\cite{Rota1996}, where it is explained how they arise in lattice theory and are connected with logic. The authors agree with the referee's view
that Rota would have enjoyed the connection with OBSs.

We conclude this section by stating a result by Dubreil and Dubreil-Jacotin~\cite{ddj} that gives an alternative characterisation of partition lattices whose equivalence relations commute. We will say that two partitions $\Pi_1$ and $\Pi_2$ are \emph{independent}, if for every $P_1 \in \Pi_1$ and $P_2 \in \Pi_2$ we have $P_1 \cap P_2 \ne \emptyset$.

\begin{prop}
Two partitions $\Pi_1$ and $\Pi_2$ commute if and only if for every part $P\in \Pi_1\vee \Pi_2$, the restrictions $\Pi_1|_P$ and $\Pi_2|_P$ are independent.
\end{prop}

\subsection{Crossing and  nesting} 

Two methods of constructing new OBSs from old, both widely used in experimental
design, are crossing and nesting, defined as follows.

Let $\mathcal{P}_1=(\Omega_1,\mathcal{B}_1)$ and
$\mathcal{P}_2=(\Omega_2,\mathcal{B}_2)$ be orthogonal
block structures.
We think of the elements of $\mathcal{B}_1$ and
$\mathcal{B}_2$ as equivalence relations. In  each construction, we build
a new OBS on $\Omega_1\times\Omega_2$.
For each pair $R_1\in\mathcal{B}_1$ and $R_2\in\mathcal{B}_2$, we define a
relation $R_1\times R_2$ to hold between two pairs $(\alpha_1,\alpha_2)$ and
$(\beta_1,\beta_2)$ if and only if $(\alpha_1,\beta_1)\in R_1$ and
$(\alpha_2,\beta_2)\in R_2$.
It is clear that $R_1\times R_2$ is an equivalence relation.

The first method uses the set of equivalence relations
\[ \{ R_1 \times R_2: R_1 \in \mathcal{B}_1,\ R_2\in\mathcal{B}_2\}.\]
This gives the set $\mathcal{B}_1 \times \mathcal{B}_2$ of
equivalence relations on $\Omega_1 \times \Omega_2$.
This is called \textit{crossing} $\mathcal{P}_1$ and $\mathcal{P}_2$, and
written $\mathcal{P}_1\times\mathcal{P}_2$.

The second method uses the set of equivalence relations
\[
\{R_1 \times U_2: R_1\in \mathcal{B}_1\} \cup \{E_1\times R_2: R_2 \in \mathcal{B}_2\},
\]
where $U_2$ is the universal relation in $\Omega_2$ and $E_1$ is the equality
relation in~$\Omega_1$.  This is called \textit{nesting} $\mathcal{P}_2$ within
$\mathcal{P}_1$, and written as $\mathcal{P}_1/\mathcal{P}_2$.

Of course,  the roles of $\mathcal{P}_1$ and $\mathcal{P}_2$ can be reversed, to give
$\mathcal{P}_2/\mathcal{P}_1$, with $\mathcal{P}_1$ nested within $\mathcal{P}_2$.

It is straightforward to show that,
if $\mathcal{P}_1$ and $\mathcal{P}_2$ are both closed under taking suprema and taking
infima, then so are $\mathcal{P}_1 \times \mathcal{P}_2$, $\mathcal{P}_1/\mathcal{P}_2$
and $\mathcal{P}_2/\mathcal{P}_1$.

If $R_1$ and $R_3$ are in $\mathcal{B}_1$ and $R_2$ and $R_4$ are in $\mathcal{B}_2$ then
$(R_1\circ R_3) \times (R_2\circ R_4)= (R_1\times R_2) \circ (R_3 \times R_4)$.
Therefore, since every two equivalence relations in $\mathcal{B}_1$ commute
 and every two equivalence relations in $\mathcal{B}_2$ commute, then the same is true for every two
 equivalence relations in each of $\mathcal{P}_1 \times \mathcal{P}_2$,
 $\mathcal{P}_1/\mathcal{P}_2$ and $\mathcal{P}_2/\mathcal{P}_1$.
\label{crossandnest}

Let $G\leq {\rm Sym}(\Gamma)$ and $H\leq {\rm Sym}(\Delta)$ be transitive permutation groups. Recall that the product action of the direct product $G\times H$ on $\Gamma\times \Delta$ is the action defined by $(\gamma, \delta)(g, h) = (\gamma g, \delta h)$ for all $(\gamma, \delta)\in \Gamma\times \Delta$ and $(g, h)\in G\times H$. For permutation group theorists, note the similarities between crossing and
nesting on one hand, and direct product (with product action) and wreath
product (with imprimitive action) on the other. Statisticians call the
results of crossing and nesting trivial OBSs \emph{row-column structures}
and \emph{block structures} respectively.

Nelder \cite{JAN1} introduced the class of orthogonal block structures which can
be obtained from trivial structures by repeatedly crossing and nesting, and
called them \emph{simple orthogonal block structures}.   See Section~\ref{sec:hist}.

\subsection{Poset block structures}\label{s:pbs}

There is a class of OBSs, more general than the simple ones, effectively
introduced in \cite{ARL}, and now called poset block structures, which
we define.

A \emph{poset block structure} is an orthogonal block structure
in which the lattice of partitions is distributive. (We have seen in
Proposition~\ref{p:dist-l} that the
distributive law is stronger than the modular law.)

Using the Fundamental Theorem on Distributive Lattices (Theorem~\ref{t:ftdl}),
we can turn this abstract definition into something more useful. Recall that
a distributive lattice $L$ is the lattice of down-sets in a poset
$(M,\sqsubseteq)$, where $M$ can be recovered from $L$ as the set of non-zero
join-indecomposable elements (that is, JI elements different from $E$). Put
$N=|M|$. Now we attach a finite set $\Omega_i$ of size $n_i>1$ to each
element $m_i\in M$, and
take $\Omega$ to be the Cartesian product of the sets $\Omega_i$ for all
$m_i\in M$. Now we need to define a partition $\Pi_D$ for each down-set $D$
in $M$. This is done as follows. Define a relation $R_D$ on $\Omega$ by
\[R_D((\alpha_1,\ldots,\alpha_N),(\beta_1,\ldots,\beta_N))\Leftrightarrow(\forall m_i\notin D)
(\alpha_i=\beta_i),\]
where $\alpha_i,\beta_i\in\Omega_i$ for all $m_i\in M$. Then $R_D$ is an
equivalence relation on $\Omega$, and we let $\Pi_D$ be the corresponding
partition. (The appearance of the poset $M$ explains the name \emph{poset
block structures}.)

It is straightforward to check that 
\begin{enumerate}
\item the partitions $E$ and $U$ of $\Omega$ correspond to the
empty set and the whole of $M$;
\item if $D_1$ and $D_2$ are down-sets in $M$, then
\[\Pi_{D_1\cap D_2}=\Pi_{D_1}\wedge\Pi_{D_2}\hbox{ and }
\Pi_{D_1\cup D_2}=\Pi_{D_1}\vee\Pi_{D_2}.\]
So the partitions $\Pi_D$ form a lattice isomorphic to the given lattice $L$.
\end{enumerate}

This is proved in \cite{DCC,TPSRAB}, where it is shown that every poset block
structure (according to our definition) is given by this construction.

At this point, we mention a paper by Yan~\cite{yan}, whose title suggests
that it concerns distributive lattices of commuting equivalence relations.
In fact, both her hypotheses and her conclusion are much stronger than
ours. In the case of uniform partitions, her theorem asserts the following:
if $\Pi_1$ and $\Pi_2$ are commuting uniform equivalence relations such
that every equivalence relation $\Psi$ which commutes with both of them
associates with them, in the sense that
\[\Psi\wedge(\Pi_1\vee\Pi_2)=(\Psi\wedge\Pi_1)\vee(\Psi\wedge\Pi_2),\]
then $\Pi_1$ and $\Pi_2$ are comparable in the partial order. (This does not
say that every distributive lattice of commuting partitions is a chain.)

\paragraph{Notation} For every $i\in \{1, \ldots, N\}$, let $A(i)$ denote the set $\{j \in \{1, \ldots, N\} \, : \, m_i \sqsubset m_j\}$ and $A[i]$ the set $\{j\in \{1, \ldots, N\} \, : \, m_i \sqsubseteq m_j\}$. Similarly, let
$D(i)$ denote the set $\{j \in \{1, \ldots, N\} \, : \, m_j \sqsubset m_i\}$ and $D[i]$ the set $\{j\in \{1, \ldots, N\} \, : \, m_j \sqsubseteq m_i\}$.
(Mnemonic: $A={}$`ancestor', $D={}$`descendant'.)

\subsection{Generalised wreath products}
\label{s:gwp}

Closely related to poset block structures is the notion of generalised
wreath product. We now define those, following the notation used in~\cite{bprs}.

\medskip

We write $\Omega^i$ for the Cartesian product $\prod_{j \in A(i)} \Omega_j$ and $\pi^i$ for the natural projection from $\Omega$ onto $\prod_{j\in A(i)}\Omega_j$. Finally, for every $m_i\in M$, let $G(m_i)$ be a permutation group on $\Omega_i$, and let $F_i$ denote the set of all functions from $\Omega^i$ into $G(m_i)$. Thus, if $f_i\in F_i$, then
$f_i$ allocates a permutation in $G(m_i)$ to each element of $\Omega^i$.

The \emph{generalised wreath product} $G$ of the groups $G(m_1), \ldots, G(m_N)$ over the poset $M$ is the group $\prod_{i = 1}^N F_i$, and it acts on $\Omega$ in the following way: if $\omega = (\omega_1,\ldots,\omega_N) \in \Omega$ and $f = (f_1, \ldots, f_N) \in G$, then
\[
(\omega f)_i = \omega_i(\omega\pi^i f_i)
\]
for $i=1,\ldots,N$.

We note that, if $M$ is the $2$-element antichain $\{m_1,m_2\}$, then the
generalised wreath product of $G(m_1)$ and $G(m_2)$ is their direct product;
while if $M$ is a $2$-element chain, with $m_1\sqsubset m_2$, then $G$ is the
wreath product $G(m_1)\wr G(m_2)$, in its imprimitive action.

The next result gives the automorphism group of a poset block structure.

\begin{prop}
The automorphism group of the poset block structure given above is the 
generalised wreath product of symmetric groups $S_{n_i}$ over the poset $(M,\sqsubseteq)$.
\label{p:pbs}
\end{prop}

This is proved in \cite{bprs}.

The operations of crossing and nesting preserve the class of poset block
structures: crossing corresponds to taking the disjoint union of the two
posets (with no comparability between them); nesting corresponds to taking
the ordered sum (with every element of the second poset below every element
of the first).

Proposition~\ref{p:pbs} shows that poset block structures always have large
auto\-morphism groups. By contrast, orthogonal block structures may have no
non-trivial automorphisms at all. Let $L$ be a Latin square, with $\Omega$
the set of positions. Take the two trivial partitions and the three partitions
into rows, columns and entries. Automorphisms of this structure are known as
\emph{autotopisms} in the Latin square literature; it is known that almost all
Latin squares have trivial autotopism group: see \cite{PJCarx,MWan}.

\section{History in Design of Experiments}
\label{sec:hist}

These ideas were developed gradually in the early days of design of statistical experiments.
In order to describe them in a standard way, we will use some notation introduced by 
Nelder in \cite{JAN1}.  If $n$ is a positive integer, then we denote by 
$\underline{\underline{n}}$
any set of  size~$n$ which has the \emph{trivial block structure} $\{U,E\}$.
(This notation is used in \cite{bailey:as} but is replaced by $[n]$ in \cite{DCC}.)

\subsection{Fisher and Yates at Rothamsted}

Ronald Fisher was the first statistician at Rothamsted Experimental Station, working there
from 1919 to 1933: see \cite{RABHist}.  He advocated two, fairly simple, blocking structures.
In the first, called a \textit{block design}, the $bk$ plots were partitioned into $b$~blocks
of size~$k$, thus giving the orthogonal block structure
$\underline{\underline{b}}/\underline{\underline{k}}$.  In the second, called a
\textit{Latin square}, the $n^2$ plots formed a square array with $n$~rows and
$n$~columns, to which $n$~treatments were applied in such a way that each treatment
occurred once in each row and once in each column.  Ignoring the treatments, this gives the
orthogonal block structure $\underline{\underline{n}}\times\underline{\underline{n}}$.  

Frank Yates worked in the Statistics Department at Rothamsted Experimental Station from
1931 until 1968: see \cite{RABHist}.  He gradually developed more and more complicated
block structures for designed experiments. His paper on ``Complex Experiments''
\cite{FYCE}, read to the Royal Statistical Society in 1935, covers many of these.
After describing block designs and Latin squares, he proposes ``splitting of plots''
(page~197) into subplots in both cases.
If the  number of subplots per plot is~$s$, this leads to the orthogonal block structures
$\underline{\underline{b}}/\underline{\underline{k}}/\underline{\underline{s}}$ and
$(\underline{\underline{n}}\times\underline{\underline{n}})/\underline{\underline{s}}$
(treatments are ignored in these block structures).
These are all based on partially ordered sets (although he did not use this terminology),
as shown in Figure~\ref{fig:PIC}.

\begin{figure}
  \setlength{\unitlength}{1mm}
  \begin{center}
    \begin{tabular}{lc@{\qquad}c@{\qquad}}
    \hline
    \centering Verbal description &
    \multicolumn{1}{p{1.2in}}{\centering Hasse diagram of poset} &
    \multicolumn{1}{p{1.3in}}{\centering Hasse diagram of OBS}\\
    \hline
\raisebox{0.55in}{Block design} &
\begin{picture}(30,30)
      \put(15,10){\blobb}
      \put(15,20){\blobb}
      \put(15,10){\line(0,1){10}}
      \put(16,10){\makebox(0,0)[l]{$\underline{\underline{k}}$}}
            \put(16,20){\makebox(0,0)[l]{$\underline{\underline{b}}$}}
\end{picture} &
\begin{picture}(30,30)
  \put(15,5){\blob}
  \put(15,15){\blob}
  \put(15,25){\blob}
  \put(16,25){\makebox(0,0)[l]{$U$}}
  \put(14,25){\makebox(0,0)[r]{$1$}}
    \put(16,15){\makebox(0,0)[l]{blocks}}
    \put(14,15){\makebox(0,0)[r]{$b$}}
      \put(16,5){\makebox(0,0)[l]{plots}}
      \put(14,5){\makebox(0,0)[r]{$bk$}}
      \put(15,5.5){\line(0,1){9}}
      \put(15,15.5){\line(0,1){9}}
  \end{picture}
\\
    \hline
    \raisebox{0.55in}{Latin square} &
    \begin{picture}(30,30)
      \put(10,15){\blobb}
      \put(20,15){\blobb}
      \put(9,15){\makebox(0,0)[r]{$\underline{\underline{n}}$}}
            \put(21,15){\makebox(0,0)[l]{$\underline{\underline{n}}$}}
    \end{picture} &
    \begin{picture}(30,30)
      \put(15,5){\blob}
      \put(15,25){\blob}
      \put(5,15){\blob}
      \put(25,15){\blob}
      \put(15.5,5.5){\line(1,1){9}}
      \put(14.5,5.5){\line(-1,1){9}}
      \put(5.5,15.5){\line(1,1){9}}
            \put(24.5,15.5){\line(-1,1){9}}
            \put(17,25){\makebox(0,0)[l]{$U$}}
            \put(13,25){\makebox(0,0)[r]{$1$}}
            \put(17,4.5){\makebox(0,0)[l]{plots}}
           \put(13,5){\makebox(0,0)[r]{$n^2$}}
           \put(26,16){\makebox(0,0)[l]{$n$}}
           \put(24,13){\makebox(0,0)[l]{columns}}
           \put(4,16){\makebox(0,0)[r]{$n$}}
           \put(6,13){\makebox(0,0)[r]{rows}}
    \end{picture}
    \\
    \hline
    \raisebox{0.7in}{Split-plot design} &
    \begin{picture}(30,40)
      \put(15,10){\blobb}
      \put(15,20){\blobb}
      \put(15,30){\blobb}
      \put(15,10){\line(0,1){20}}
      \put(16,30){\makebox(0,0)[l]{$\underline{\underline{b}}$}}
      \put(16,20){\makebox(0,0)[l]{$\underline{\underline{k}}$}}
      \put(16,10){\makebox(0,0)[l]{$\underline{\underline{s}}$}}
    \end{picture}
    &
    \begin{picture}(30,40)
      \put(15,5){\blob}
      \put(15,15){\blob}
      \put(15,25){\blob}
      \put(15,35){\blob}
      \put(15,5.5){\line(0,1){9}}
      \put(15,15.5){\line(0,1){9}}
      \put(15,25.5){\line(0,1){9}}
      \put(16,35){\makebox(0,0)[l]{$U$}}
      \put(16,25){\makebox(0,0)[l]{blocks}}
      \put(16,15){\makebox(0,0)[l]{plots}}
      \put(16,5){\makebox(0,0)[l]{subplots}}
      \put(14,35){\makebox(0,0)[r]{$1$}}
      \put(14,25){\makebox(0,0)[r]{$b$}}
      \put(14,15){\makebox(0,0)[r]{$bk$}}
      \put(14,5){\makebox(0,0)[r]{$bks$}}
    \end{picture}
    \\
    \hline
    \raisebox{0.7in}{\parbox{1.2in}{\raggedright Latin square with split plots}}
    &
    \begin{picture}(30,40)
      \put(15,15){\blobb}
      \put(25,25){\blobb}
      \put(5,25){\blobb}
      \put(15,15){\line(1,1){10}}
      \put(15,15){\line(-1,1){10}}
      \put(26,25){\makebox(0,0)[l]{$\underline{\underline{n}}$}}
      \put(4,25){\makebox(0,0)[r]{$\underline{\underline{n}}$}}
      \put(15,14){\makebox(0,0)[t]{$\underline{\underline{s}}$}}
    \end{picture}
    &
    \begin{picture}(30,40)
      \put(15,5){\blob}
        \put(15,15){\blob}
        \put(25,25){\blob}
        \put(5,25){\blob}
        \put(15,35){\blob}
        \put(15,5.5){\line(0,1){9}}
        \put(15.5,15.5){\line(1,1){9}}
        \put(14.5,15.5){\line(-1,1){9}}
        \put(5.5,25.5){\line(1,1){9}}
        \put(24.5,25.5){\line(-1,1){9}}
        \put(17,35){\makebox(0,0)[l]{$U$}}
        \put(16,4.5){\makebox(0,0)[l]{subplots}}
        \put(16,14.5){\makebox(0,0)[l]{plots}}
           \put(14,15){\makebox(0,0)[r]{$n^2$}}
           \put(14,5){\makebox(0,0)[r]{$n^2s$}}
                      \put(26,26){\makebox(0,0)[l]{$n$}}
           \put(24,23){\makebox(0,0)[l]{columns}}
           \put(4,26){\makebox(0,0)[r]{$n$}}
           \put(6,23){\makebox(0,0)[r]{rows}}
    \end{picture}
    \\
    \hline
  \end{tabular}
  \end{center}
  \caption{Orthogonal block structures mentioned by Yates in \cite{FYCE}
    \label{fig:PIC}}
  \end{figure}

Yates also suggests ``two $4\times 4$ Latin squares with subplots'' (page~201), which gives
the orthogonal block structure
$\underline{\underline{2}}/(\underline{\underline{4}} \times \underline{\underline{4}})
/\underline{\underline{2}}$;
splitting each row of an $\underline{\underline{r}}\times \underline{\underline{c}}$
rectangle into two subrows, which gives the orthogonal block structure
$(\underline{\underline{r}}/\underline{\underline{2}})\times \underline{\underline{c}}$
(page 202);
and a collection of four $5 \times 5$ Latin squares (page 218), which gives the orthogonal
block structure
$\underline{\underline{4}}/(\underline{\underline{5}}\times\underline{\underline{5}})$.
These are shown in Figure~\ref{fig:PIC2}.

\begin{figure}
  \setlength{\unitlength}{1mm}
  \begin{center}
    \begin{tabular}{lc@{\qquad}c@{\qquad}}
    \hline
   \centering Verbal description &
    \multicolumn{1}{p{1.2in}}{\centering Hasse diagram of poset} &
    \multicolumn{1}{p{1.7in}}{\centering Hasse diagram of OBS}\\
    \hline
    \raisebox{0.95in}{\parbox{1.2in}{\raggedright Two Latin squares with subplots}} &
    \begin{picture}(30,50)
      \put(25,25){\blobb}
      \put(15,15){\blobb}
      \put(5,25){\blobb}
      \put(15,35){\blobb}
      \put(15,15){\line(1,1){10}}
      \put(15,15){\line(-1,1){10}}
            \put(5,25){\line(1,1){10}}
            \put(25,25){\line(-1,1){10}}
      \put(26,25){\makebox(0,0)[l]{$\underline{\underline{4}}$}}
      \put(4,25){\makebox(0,0)[r]{$\underline{\underline{4}}$}}
      \put(15,14){\makebox(0,0)[t]{$\underline{\underline{2}}$}}
           \put(15,36){\makebox(0,0)[b]{$\underline{\underline{2}}$}}
    \end{picture}
    &
    \begin{picture}(40,50)
      \put(20,5){\blob}
      \put(20,15){\blob}
      \put(30,25){\blob}
      \put(10,25){\blob}
      \put(20,35){\blob}
      \put(20,45){\blob}
      \put(20,5.5){\line(0,1){9}}
\put(20.5,15.5){\line(1,1){9}}
\put(20,35.5){\line(0,1){9}}
\put(19.5,15.5){\line(-1,1){9}}
\put(10.5,25.5){\line(1,1){9}}
\put(29.5,25.5){\line(-1,1){9}}
\put(21,45){\makebox(0,0)[l]{$U$}}
\put(19,45){\makebox(0,0)[r]{$1$}}
\put(21,35){\makebox(0,0)[l]{squares}}
\put(19,35.7){\makebox(0,0)[r]{$2$}}
\put(21,15){\makebox(0,0)[l]{plots}}
\put(19,15.1){\makebox(0,0)[r]{$32$}}
\put(21,4.5){\makebox(0,0)[l]{subplots}}
\put(19,4.75){\makebox(0,0)[r]{$64$}}
\put(9,25){\makebox(0,0)[r]{$8$ rows}}
\put(31,25){\makebox(0,0)[l]{$8$ columns}}
    \end{picture}
    \\
    \hline
    \raisebox{0.7in}{\parbox{1.2in}{Splitting rows of a rectangle}} &
    \begin{picture}(30,40)
      \put(10,15){\blobb}
      \put(10,25){\blobb}
      \put(20,20){\blobb}
      \put(10,15){\line(0,1){10}}
      \put(9,25){\makebox(0,0)[r]{$\underline{\underline{r}}$}}
      \put(9,15){\makebox(0,0)[r]{$\underline{\underline{2}}$}}
       \put(21,20){\makebox(0,0)[l]{$\underline{\underline{c}}$}}
      \end{picture}
    &
    \begin{picture}(40,40)
      \put(25,35){\blob}
      \put(15,25){\blob}
      \put(35,25){\blob}
      \put(5,15){\blob}
      \put(25,15){\blob}
      \put(15,5){\blob}
      \put(15.5,5.5){\line(1,1){9}}
      \put(14.5,5.5){\line(-1,1){9}}
      \put(25.5,15.5){\line(1,1){9}}
      \put(5.5,15.5){\line(1,1){9}}
      \put(15.5,25.5){\line(1,1){9}}
      \put(24.5,15.5){\line(-1,1){9}}
      \put(34.5,25.5){\line(-1,1){9}}
      \put(27,35){\makebox(0,0)[l]{$U$}}
      \put(23,35){\makebox(0,0)[r]{$1$}}
      \put(13,25){\makebox(0,0)[r]{$r$ rows}}
      \put(5,17){\makebox(0,0)[r]{$2r$}}
      \put(6,13){\makebox(0,0)[r]{subrows}}
      \put(14,4.75){\makebox(0,0)[r]{$2rc$}}
      \put(16,4.5){\makebox(0,0)[l]{subplots}}
      \put(24,13){\makebox(0,0)[l]{$rc$ plots}}
      \put(34,23){\makebox(0,0)[l]{columns}}
      \put(36,26){\makebox(0,0)[l]{$c$}}
      \end{picture}
    \\
    \hline
    \raisebox{0.7in}{\parbox{1.25in}{\raggedright Four Latin squares of order five}} &
    \begin{picture}(30,40)
      \put(5,15){\blobb}
      \put(25,15){\blobb}
      \put(15,25){\blobb}
      \put(15,26){\makebox(0,0)[b]{$\underline{\underline{4}}$}}
      \put(5,14){\makebox(0,0)[t]{$\underline{\underline{5}}$}}
      \put(25,14){\makebox(0,0)[t]{$\underline{\underline{5}}$}}
      \put(15,25){\line(1,-1){10}}
        \put(15,25){\line(-1,-1){10}}
    \end{picture}
    &
    \begin{picture}(40,40)
      \put(20,35){\blob}
      \put(20,25){\blob}
      \put(10,15){\blob}
      \put(30,15){\blob}
      \put(20,5){\blob}
      \put(20.5,5.5){\line(1,1){9}}
      \put(19.5,5.5){\line(-1,1){9}}
       \put(10.5,15.5){\line(1,1){9}}
       \put(29.5,15.5){\line(-1,1){9}}
       \put(20,25.5){\line(0,1){9}}
       \put(21,35){\makebox(0,0)[l]{$U$}}
       \put(19,35){\makebox(0,0)[r]{$1$}}
       \put(21,25){\makebox(0,0)[l]{squares}}
\put(19,25.7){\makebox(0,0)[r]{$4$}}
\put(9,15){\makebox(0,0)[r]{$20$ rows}}
\put(31,16){\makebox(0,0)[l]{$20$}}
\put(29,13){\makebox(0,0)[l]{columns}}
      \put(19,4.75){\makebox(0,0)[r]{$100$}}
      \put(21,4.5){\makebox(0,0)[l]{plots}}
    \end{picture}
       \\
    \hline
    \end{tabular}
  \end{center}
  \caption{More orthogonal block structures mentioned by Yates \label{fig:PIC2}}
\end{figure}

\subsection{Nelder's simple orthogonal block structures}
John Nelder worked in the Statistics Section of the UK's National Vegetable Research Station
from 1951 to 1968.  In two papers \cite{JAN1,JAN2} in 1965 he introduced the class of
orthogonal block structures which can be obtained from trivial structures by repeated
crossing and nesting, and called them \textit{simple orthogonal block structures}.
In that year, he also visited CSIRO (the Commonwealth Scientific and Industrial Research
Organisation) at the Waite Campus of the University of Adelaide in South Australia, where
he worked with Graham Wilkinson to start developing the statistical software GenStat. He
and colleagues developed GenStat further while  he was Head of the Statistics Department at
Rothamsted Experimental Station from 1968 to 1984.
The benefit of iterated crossing and nesting is that each block structure can be described
by a simple formula, which can be input as a line in the program used to analyse the data
obtained from an experiment.

\subsection{Statisticians at Iowa State University}

In parallel with Nelder's work was the work of Oscar Kempthorne and his colleagues.
Kempthorne worked at the  Statistics Department at Rothamsted Experimental Station from
1941 to 1946.  He spent most of the rest of his career at Iowa State University.  While there,
he obtained a grant from the Aeronautical Research Laboratory to work with his colleagues
on various problems in the design of experiments.

Their technical report \cite{ARL} was completed in November 1961, and consisted
of 218 typed pages.
It uses the phrases ``experimental structure'' and ``response structure'' for what we call
``block structure''.   Sometimes the treatments were also included in this structure.
Chapter~3 is based on the PhD theses of
Zyskind \cite{Zys} and Throckmorton \cite{Throck}; part of this was later
published as \cite{Zys2}.

With hindsight, it seems that they were trying to define poset block structures,
but they managed to confuse the poset~$M$ of coordinates with the lattice of
partitions.  They denoted the universal partition~$U$ by $\mu$, and the
equality partition~$E$ by $\varepsilon$.  They used complicated formulae,
called \textit{symbolic representations}, to
explain the poset~$M$, but then included $\mu$ and
$\varepsilon$ in the corresponding Hasse diagram, which they called the
\textit{structure diagram}.  They dealt with all posets of size at most four,
and showed $16$ of the $63$ posets of size five.

Figure~\ref{fig:ARL} shows three of their block structures.  The first of
these is also in Figure~\ref{fig:PIC2}; the last one
cannot be obtained by crossing and nesting, so it needs two formulae.

\begin{figure}
  \setlength{\unitlength}{1mm}
  \begin{center}
    \begin{tabular}{lc@{\qquad}c@{\qquad}}
    \hline
    \multicolumn{1}{p{1in}}{\centering Symbolic representation} &
    \multicolumn{1}{p{1.2in}}{\centering Structure diagram} &
    \multicolumn{1}{p{1.7in}}{\centering Hasse diagram of OBS}\\
    \hline
    \raisebox{0.7in}{\parbox{1.2in}{$S\colon RC$}} &
    \begin{picture}(30,40)
      \put(15,35){\bloz}
      \put(15,25){\bloz}
      \put(5,15){\bloz}
      \put(25,15){\bloz}
      \put(15,5){\bloz}
      \put(16,6){\line(1,1){8}}
      \put(14,6){\line(-1,1){8}}
        \put(6,16){\line(1,1){8}}
        \put(24,16){\line(-1,1){8}}
        \put(15,26){\line(0,1){8}}
        \put(15,37){\makebox(0,0)[b]{$\mu$}}
        \put(17,25){\makebox(0,0)[l]{$S$}}
        \put(27,15){\makebox(0,0)[l]{$C$}}
        \put(3,15){\makebox(0,0)[r]{$R$}}
        \put(15,3){\makebox(0,0)[t]{$\varepsilon$}}
    \end{picture} &
        \begin{picture}(40,40)
      \put(20,35){\blob}
      \put(20,25){\blob}
      \put(10,15){\blob}
      \put(30,15){\blob}
      \put(20,5){\blob}
      \put(20.5,5.5){\line(1,1){9}}
      \put(19.5,5.5){\line(-1,1){9}}
       \put(10.5,15.5){\line(1,1){9}}
       \put(29.5,15.5){\line(-1,1){9}}
       \put(20,25.5){\line(0,1){9}}
       \put(20,36){\makebox(0,0)[b]{$U$}}
             \put(22,25){\makebox(0,0)[l]{$S$}}
             \put(9,15){\makebox(0,0)[r]{$R$}}
             \put(31,15){\makebox(0,0)[l]{$C$}}
      \put(20,4){\makebox(0,0)[t]{$E$}}
    \end{picture}
    \\
    \hline
    \raisebox{0.9in}{\parbox{1in}{$S\colon(R)(C\colon L)$}} &
    \begin{picture}(30,50)
      \put(15,45){\bloz}
      \put(15,35){\bloz}
      \put(5,25){\bloz}
      \put(5,15){\bloz}
      \put(25,15){\bloz}
      \put(15,5){\bloz}
      \put(16,6){\line(1,1){8}}
      \put(14,6){\line(-1,1){8}}
      \put(5,16){\line(0,1){8}}
      \put(6,26){\line(1,1){8}}
      \put(24.8,16.4){\line(-1,2){9}}
      \put(15,36){\line(0,1){8}}
       \put(15,47){\makebox(0,0)[b]{$\mu$}}
        \put(17,35){\makebox(0,0)[l]{$S$}}
        \put(27,15){\makebox(0,0)[l]{$R$}}
        \put(3,15){\makebox(0,0)[r]{$L$}}
        \put(3,25){\makebox(0,0)[r]{$C$}}
        \put(15,3){\makebox(0,0)[t]{$\varepsilon$}}
    \end{picture}
    &
    \begin{picture}(40,50)
      \put(25,45){\blob}
      \put(25,35){\blob}
      \put(15,25){\blob}
      \put(35,25){\blob}
      \put(5,15){\blob}
      \put(25,15){\blob}
      \put(15,5){\blob}
      \put(25,35.5){\line(0,1){9}}
      \put(15.5,5.5){\line(1,1){9}}
      \put(14.5,5.5){\line(-1,1){9}}
      \put(25.5,15.5){\line(1,1){9}}
      \put(5.5,15.5){\line(1,1){9}}
      \put(15.5,25.5){\line(1,1){9}}
      \put(24.5,15.5){\line(-1,1){9}}
      \put(34.5,25.5){\line(-1,1){9}}
      \put(27,35){\makebox(0,0)[l]{$S$}}
      \put(25,46){\makebox(0,0)[b]{$U$}}
      \put(13,25){\makebox(0,0)[r]{$C$}}
      \put(4,15){\makebox(0,0)[r]{$L$}}
      \put(15,4){\makebox(0,0)[t]{$E$}}
      \put(24,13){\makebox(0,0)[l]{$C\wedge R$}}
         \put(36,25){\makebox(0,0)[l]{$R$}}
      \end{picture}
    \\
    \hline
    \raisebox{0.9in}{\parbox{1in}{\raggedright $(S\colon Q)$ $(P)$ and $(SP\colon R)$}}
    &
    \begin{picture}(30,50)
      \put(15,40){\bloz}
      \put(5,30){\bloz}
      \put(25,30){\bloz}
      \put(5,20){\bloz}
      \put(25,20){\bloz}
      \put(15,10){\bloz}
      \put(16,11){\line(1,1){8}}
      \put(14,11){\line(-1,1){8}}
      \put(6,31){\line(1,1){8}}
      \put(24,31){\line(-1,1){8}}
      \put(5,21){\line(0,1){8}}
      \put(25,21){\line(0,1){8}}
      \put(6,21){\line(2,1){18}}
      \put(15,42){\makebox(0,0)[b]{$\mu$}}
      \put(15,8){\makebox(0,0)[t]{$\varepsilon$}}
      \put(3,30){\makebox(0,0)[r]{$P$}}
      \put(3,20){\makebox(0,0)[r]{$R$}}
      \put(27,30){\makebox(0,0)[l]{$S$}}
      \put(27,20){\makebox(0,0)[l]{$Q$}}
    \end{picture}
    &
    \begin{picture}(40,50)
      \put(15,45){\blob}
      \put(5,35){\blob}
      \put(25,35){\blob}
       \put(15,25){\blob}
       \put(35,25){\blob}
        \put(5,15){\blob}
        \put(25,15){\blob}
        \put(15,5){\blob}
        \put(15.5,5.5){\line(1,1){9}}
        \put(25.5,15.5){\line(1,1){9}}
        \put(5.5,15.5){\line(1,1){9}}
        \put(15.5,25.5){\line(1,1){9}}
        \put(5.5,35.5){\line(1,1){9}}
        \put(14.5,5.5){\line(-1,1){9}}
        \put(24.5,15.5){\line(-1,1){9}}
        \put(34.5,25.5){\line(-1,1){9}}
        \put(14.5,25.5){\line(-1,1){9}}
        \put(24.5,35.5){\line(-1,1){9}}
        \put(15,46){\makebox(0,0)[b]{$U$}}
        \put(4,35){\makebox(0,0)[r]{$P$}}
        \put(27,35){\makebox(0,0)[l]{$S$}}
        \put(37,25){\makebox(0,0)[l]{$Q$}}
        \put(16,25){\makebox(0,0)[l]{$P\wedge S$}}
        \put(4,15){\makebox(0,0)[r]{$R$}}
        \put(25,13){\makebox(0,0)[l]{$P\wedge Q$}}
        \put(15,4){\makebox(0,0)[t]{$E$}}
      \end{picture}
    \\
    \hline
    \end{tabular}
  \end{center}
  \caption{Some orthogonal block structures in \cite{ARL}\label{fig:ARL}}
    \end{figure}

\subsection{Unifying the theory}
\label{s:unify}

In \cite{TPSRAB}, Speed and Bailey aimed to combine the two approaches by explaining
Nelder's ``simple orthogonal block structures'' and Throckmorton's  ``complete balanced
block structures'' as ``association schemes derived from finite distributive lattices of
commuting uniform equivalence relations''.  They noted that the words ``permutable'' and
``permuting'' were sometimes used  in place of  ``commuting''.  Each partition is
defined by a ``hereditary'' subset of the poset~$M$. 
This is the dual notion to down-set.  A subset~$H$ of $M$ is hereditary if, whenever $m\in H$
and $m\sqsubseteq m'$, then $m'\in H$.  Then $\Omega = \Omega_1 \times \cdots \times
\Omega_N$ (where $N=\left|M\right|$).  Two elements $(\alpha_1, \ldots, \alpha_N)$ and
$(\beta_1, \ldots, \beta_N)$ are in the same part of the partition $\Pi_H$ if and only if
$\alpha_i=\beta_i$ for all $i$ in $H$.

To match the partial order on partitions to the partial order $\subseteq$ on subsets of~$M$,
they defined $\preccurlyeq$ in the opposite way to what we do here. They 
proved that every distributive block structure is isomorphic to a poset block structure, but did
not use the latter term, even though they showed that the construction depends on a
partially ordered set.

They also explained that most of the theory extends to what we now call an \textit{orthogonal
  block structure}, where the lattice is modular but not necessarily distributive.
Figure~\ref{fig:SB} shows the corresponding Hasse diagrams in their two examples.
In the one on the left, the non-trivial partitions form the rows, columns, and the Greek and Latin letter partitions of a pair of mutually orthogonal Latin squares. Note that the underlying set has
size~$n^2$ with $n \notin \{1,2,6\}$, since it is well known that there exists a pair of two mutually orthogonal Latin squares of order $n$ if and only if $n\notin \{1,  2, 6\}$.  One way of achieving the one on the right is to use
some carefully chosen subgroups of the elementary abelian group of order~$16$.

\begin{figure}
  \setlength{\unitlength}{1mm}
  \begin{center}
       \begin{tabular}{cc}
 \hline
         \begin{picture}(40,50)
      \put(20,15){\blob}
      \put(20,35){\blob}
      \put(5,25){\blob}
      \put(15,25){\blob}
      \put(25,25){\blob}
      \put(35,25){\blob}
      \put(20.5,15.5){\line(3,2){14}}
      \put(20.3,15.2){\line(1,2){4.5}}
       \put(19.5,15.5){\line(-3,2){14}}
       \put(19.7,15.2){\line(-1,2){4.5}}
        \put(20.5,34.5){\line(3,-2){14}}
      \put(20.3,34.8){\line(1,-2){4.5}}
       \put(19.5,34.5){\line(-3,-2){14}}
       \put(19.7,34.8){\line(-1,-2){4.5}}
    \end{picture}
         &
         \begin{picture}(50,50)
           \put(25,5){\blob}
           \put(15,15){\blob}
           \put(25,15){\blob}
           \put(35,15){\blob}
             \put(5,25){\blob}
                      \put(15,25){\blob}
           \put(25,25){\blob}
           \put(35,25){\blob}
           \put(45,25){\blob}
                      \put(15,35){\blob}
           \put(25,35){\blob}
           \put(35,35){\blob}
           \put(25,45){\blob}
           \put(25.5,5.5){\line(1,1){9}}
           \put(35.5,15.5){\line(1,1){9}}
            \put(15.5,15.5){\line(1,1){9}}
            \put(25.5,25.5){\line(1,1){9}}
              \put(5.5,25.5){\line(1,1){9}}
              \put(15.5,35.5){\line(1,1){9}}
              \put(24.5,5.5){\line(-1,1){9}}
              \put(14.5,15.5){\line(-1,1){9}}
               \put(34.5,15.5){\line(-1,1){9}}
                \put(24.5,25.5){\line(-1,1){9}}
 \put(44.5,25.5){\line(-1,1){9}}
 \put(34.5,35.5){\line(-1,1){9}}
 \put(25,5.5){\line(0,1){9}}
 \put(25,15.5){\line(0,1){9}}
 \put(35,15.5){\line(0,1){9}}
 \put(15,15.5){\line(0,1){9}}
  \put(25,25.5){\line(0,1){9}}
 \put(35,25.5){\line(0,1){9}}
 \put(15,25.5){\line(0,1){9}}
  \put(25,35.5){\line(0,1){9}}
         \end{picture}
    \\
    \hline
    \end{tabular}
\end{center}
  \caption{Hasse diagrams of two non-distributive orthogonal block structures\label{fig:SB}}
\end{figure}

In \cite{Geelong}, Bailey restricted attention to  distributive block structures, using the term
``ancestral subset'' in place of ``hereditary subset'' and drawing the Hasse diagrams in the
way consistent with our current use of the refinement partial order $\preccurlyeq$.
This cited \cite{Zys} as well as \cite{Throck}, and commented that Holland \cite{Holl}
``defines the automorphism group of a poset block structure to be a
\textit{generalised wreath product}''.  The explicit form for such a group was given in \cite{bprs},
following the arguments in \cite{Holl}. 

Paper \cite{bprs} gives a formal definition of \textit{poset block structure}
and an auto\-morphism of such a structure.  It shows that, in the finite case,
the auto\-morphism group is the generalised wreath product of the relevant
symmetric groups.  The argument draws on work of Wells \cite{Well} for
semi-groups.  The paper also states that, in the finite case, the
generalised wreath product
of permutation groups is the same as that constructed by \cite{Holl, Sill}.

In \cite{SB2}, Speed and Bailey discuss \textit{factorial dispersion models}, which
are statistical models whose underlying structure is a poset block structure. Now
hereditary subsets are called \textit{filters} and the refinement partial order is shown
in the same way as we do here.

Papers \cite{Geelong,bprs,SB2} have the disadvantage that the partial order on the
subsets of~$M$ is the wrong way up for inclusion.  In the current paper, our use of
down-sets rather than hereditary subsets gets  round this problem.

In \cite{HS}, Houtman and Speed extend the meaning of ``orthogonal block structure''
to mean a particular desirable property of  covariance matrices.  This is even more general
than their being based on an association scheme, so we do not use that meaning here.

The survey paper \cite{DCC} explains the combinatorial aspects of all these
ideas in more detail.  It notes that a ``complete balanced response structure''
is not necessarily a  poset block structure, but can always be extended to one
by the inclusion of infima. 

It also discusses automorphisms.  In the present
paper, an automorphism of a poset block structure is a permutation of the
base-set $\Omega$ which preserves each of the relevant partitions.
In \cite{DCC,J}, this is called a ``strong automorphism'', while a ``weak
automorphism'' preserves the set of these partitions. These are called
``strict automorphism'' and ``automorphism'', respectively, in \cite{Par}.

If there are non-identity weak automorphisms, then under suitable conditions
we can extend our group by adjoining these. We do not discuss this here, but
note that three of the types of primitive group in the celebrated
O'Nan--Scott theorem~\cite{scott} can be realised in this way: affine groups,
wreath products with product action, and diagonal groups.

\subsection{Statistics and group theory}
Why do statisticians care about these groups?  First, because of the need to randomise.
An experimental design is an allocation of treatments to the elements of the base-set
$\Omega$. 
To avoid possible bias, this
allocation is then randomised by applying a permutation chosen at random from the
automorphism group of the block structure. Denote by $Y_\alpha$ the random variable for
the response on plot~$\alpha$.  The method of randomisation allows us to assume that the
covariance of $Y_\alpha$ and $Y_\beta$ is equal to the covariance of $Y_\gamma$ and
$Y_\delta$ (but unknown in advance) if and only if $(\gamma,\delta)$ is in the same orbit
of the action of the automorphism group on $\Omega \times \Omega$ as at least one
of $(\alpha,\beta)$ and $(\beta,\alpha)$.

For the full generalised wreath product of symmetric groups, these orbits on pairs are
precisely the association classes of the association scheme described in \cite{TPSRAB}.
Thus the eigenspaces of the covariance matrix are known in advance of data collection.
These eigenspaces are called \textit{strata} in \cite{JAN1,JAN2}.  Data can be projected onto
each stratum for a straightforward analysis.

Now suppose that each symmetric group $G_i$ in the generalised wreath product is
replaced by a subgroup $H_i$.  Lemma~11 in \cite{bprs}
shows that the eigenspaces are known
in advance if and only if the permutation character of the generalised wreath product
is multiplicity-free (or a slight weakening of this, because the covariance-matrix must
be symmetric).  In particular, so long as each subgroup $H_i$ is doubly transitive then the
strata are the same as they are for the generalised wreath product of symmetric groups.

Paper~\cite{TPSRAB} concludes with acknowledgements to several people, including
P.~J.~Cameron and D.~E.~Taylor.  These two had explained to the authors of
\cite{TPSRAB} the importance of having
a permutation character which is multiplicity-free.

\section{Permutation Groups}
\label{s:pg}

In this section, we consider transitive permutation groups, and say that such
a group $G$ has the \emph{OB property} (respectively, the \emph{PB property}) if the
$G$-invariant partitions form an orthogonal block structure (respectively, a poset
block structure). We examine the behaviour of these properties under various
products of permutation groups. Our major result is a proof that any transitive
group $G$ with the PB property is embeddable in a generalised wreath product of
transitive groups extracted from $G$. 

\subsection{Introduction to OB groups}
Let $G$ be a transitive permutation group on $\Omega$. The set of all
$G$-invariant partitions satisfies the first three of the four conditions listed in
Section~\ref{sec:OBS}
for an orthogonal block structure. When does it satisfy the fourth?
We will say that $G$ has the OB property if the fourth condition holds.

We observe that, for a given point $\alpha\in\Omega$, there is a natural
order-preserving bijection between $G$-invariant partitions of $\Omega$
and subgroups of $G$ containing $G_\alpha$: if $G_\alpha\le H\le G$, then
$\alpha H$ is a part of a $G$-invariant partition; in the other direction,
if $\Pi$ is a $G$-invariant partition, the corresponding subgroup is the 
setwise stabiliser of the part of $\Pi$ containing $\alpha$. If $\Pi_1$ and
$\Pi_2$ correspond to $H$ and $K$, then $\Pi_1\wedge\Pi_2$ corresponds to
$H\cap K$, and $\Pi_1\vee\Pi_2$ corresponds to $\langle H,K\rangle$.
(The result for join is in \cite{a-mcs}, and for meet 
\cite[Theorem 1.5A]{dm}.


\begin{theorem}
Suppose that $G$-invariant partitions $\Pi_1$ and $\Pi_2$ correspond to
subgroups $H$ and $K$ containing $G_\alpha$. Then $\Pi_1$ and $\Pi_2$
commute if and only if $HK=KH$.
\label{t:com}
\end{theorem}

\begin{proof} 
Suppose that $HK=KH$. Then $HK$ is a subgroup, and is equal to
$\langle H,K\rangle$. 
The points $\beta$ such that $(\alpha,\beta)\in\Pi_1\circ\Pi_2$ (respectively,
$\Pi_2\circ\Pi_1$, $\Pi_1\vee\Pi_2$) are those that can be reached from $\alpha$
by applying an element of $HK$ (respectively, $KH$, $\langle H,K\rangle$). So the
three relations are all equal.

Conversely, suppose that $\Pi_1$ and $\Pi_2$ are the $G$-invariant partitions
corresponding to $H$ and $K$, and that $\Pi_1\circ\Pi_2=\Pi_1\vee\Pi_2$.
In particular, this holds for the part containing $\alpha$. So any point in
this part can be reached from $\alpha$ by first moving to a point $\beta$ in
the same part of $\Pi_1$, then to a point $\gamma$ in the same part of $\Pi_2$
as $\beta$. Since the stabiliser of the part of $\Pi_1$ containing $\alpha$ is~$H$,
we have $\beta=\alpha h$ for some $h\in H$. Then the part of $\Pi_2$
containing $\beta$ is obtained by mapping the part containing $\alpha$
by $h$, so its stabiliser is $K^h$; so $\gamma=\beta h^{-1}kh$
for some $k\in K$. Thus $\gamma=\alpha kh$. We conclude that the
part of $\Pi_1\vee\Pi_2$ containing $\alpha$ is $\alpha KH$. Because
the partitions commute, this part is also equal to $\alpha HK$. We now claim that this implies that $HK=KH$.

Let $g_1\in HK$. Then $\alpha g_1\in \alpha HK = \alpha KH$, so there exists some $g_2 \in KH$ such that $\alpha g_1 = \alpha g_2$. It follows that $\alpha g_1g_2^{-1} = \alpha$, and so $g_1g_2^{-1}\in G_{\alpha}$, which in turn implies that $g_1 \in G_{\alpha}g_2 \subseteq G_{\alpha}KH$. But since $G_{\alpha} \leq H$ and $G_{\alpha} \leq K$, we have $G_{\alpha}KH = KH$, and so $g_1 \in KH$. Therefore $HK \subseteq KH$. By symmetry, we also get $KH \subseteq HK$, and thus $HK = KH$, as claimed.
\qed
\end{proof}

\begin{cor}
$G$ has the OB property if and only if, for any two subgroups $H$ and $K$
between $G_\alpha$ and $G$, we have $HK=KH$.
\label{c:permut}
\end{cor}

\begin{proof} This simply means that the conditions of Theorem~\ref{t:com} hold
for all $G$-invariant partitions (or all subgroups containing $G_\alpha$). \qed
\end{proof}

Subgroups $H$ and $K$ are said to \emph{commute} if $HK=KH$. Thus a transitive
permutation group has the OB property if any two subgroups containing a given
point stabiliser commute. (Note: In the literature the term ``permute'' is
often used for this concept; since our subject is permutation groups, we feel
that ``commute'' is less confusing.)

In some cases we can describe all the orthogonal block structures arising from
OB groups.
\begin{enumerate}
\item If the degree $n$ is prime, then a transitive permutation group of 
degree $n$ preserves only the trivial partitions, so it is OB, with the
corresponding OBS being trivial.
\item Suppose that $n=pq$, where $p$ and $q$ are distinct primes. If $G$ is
OB, then it has at most one invariant partition with parts of size $p$, and at
most one with parts of size $q$. Thus, if $G$ is imprimitive, the OBS preserved
by $G$ is obtained from the trivial structures on $p$ and $q$ points either by
crossing or by nesting in either order. Thus $G$ is embedded either in the
direct product or the wreath product (in some order) of transitive groups of
degrees $p$ and $q$.
\item Suppose that $n = p^2$ for some prime $p$ and that there are more than two non-trivial $G$-invariant partitions. Then each such partition has $p$ parts of size $p$, and any two have meet $E$ and join $U$. Thus any three of these partitions give the structure of a Latin square $L$ to the underlying set $\Omega$.

Let $P$ be a Sylow $p$-subgroup of $G$. The stabiliser in $P$ of a part of a non-trivial $G$-invariant partition $\Pi_1$ has index $p$ in $G$ by the Orbit-Stabiliser Theorem, and fixes all parts of $\Pi_1$. If $\Pi_2$ is another such partition, then the stabiliser in $P$ of a part in each of $\Pi_1$ and $\Pi_2$ fixes all parts of $\Pi_1\wedge\Pi_2$, that is, it is the identity. So $|P|=p^2$; and $P$, having more than one subgroup of index~$p$, is the elementary abelian group. Since $P$ induces a cyclic group $C_p$ of order $p$ on each of the sets of rows, columns and letters of $L$, we see that $L$ is the Cayley table of~$C_p$.

The automorphism group of this Latin square is $(C_p\times C_p)\colon C_{p-1}$. So $G$ is contained in this group. But $C_p\times C_p$ has $p+1$ non-trivial invariant partitions (corresponding to its $p+1$ subgroups of order $p$) and $C_{p-1}$ fixes just two of these partitions and permutes the other $p-1$ regularly. Since $G$ fixes at least three partitions, we conclude that $G=C_p\times C_p$. The $p+1$ non-trivial $G$-invariant partitions together with $U$ and $E$ form an orthogonal block structure which is not a poset block structure.
\end{enumerate}

(Transitive groups of degree $pq$ may not be OB. If $q\mid p-1$, then the
nonabelian group of order $pq$, acting regularly, has $p$ invariant partitions
each with $p$ parts of size $q$; these do not commute. In other words, the
subgroups of order $q$ do not commute.)

\subsection{Properties of OB groups}

\subsubsection{General results}

A transitive permutation group $G$ is \emph{pre-primitive} (see~\cite{a-mcs})
if every $G$-invariant partition is the orbit partition of a subgroup of $G$.
As explained in that paper, we may assume that this subgroup of $G$ is normal.

\begin{cor}
If $G$ is pre-primitive, then it has the OB property.
\end{cor}

\begin{proof} If $G$ is pre-primitive, then the $G$-invariant partitions
  are orbit partitions of normal subgroups of $G$; and normal subgroups commute,
  so the corresponding partitions commute.\qed
\end{proof}

Both properties can be expressed in group-theoretic terms.
Thus, the transitive permutation group $G$ is pre-primitive if and only if
$G_\alpha$ has a normal supplement in every overgroup (that is, every overgroup
has the form $N_HG_\alpha$, where $N_H$ is a normal subgroup of $G$). By
Theorem~\ref{t:com}, $G$ is OB if and only if all the subgroups containing
$G_\alpha$ commute. If $H=N_HG_\alpha$ and $K=N_KG_\alpha$, with 
$N_H, N_K$ normal in $G$, then $HK=N_HG_\alpha.N_KG_\alpha=N_HN_KG_\alpha
=N_KN_HG_\alpha$, so $HK=KH$.

\begin{cor}
Suppose that the $G$-invariant partitions form a chain under $\preccurlyeq$. Then
$G$ has the OB property.
\label{c:chain}
\end{cor}

\begin{proof}
If $\Pi_1\preccurlyeq\Pi_2$, then $\Pi_1$ and $\Pi_2$ commute.\qed
\end{proof}

A transitive permutation group $G$ is \emph{primitive} if the only
$G$-invariant partitions are the trivial ones (the partition $E$ into singletons
and the partition $U$ with a single part); it is \emph{quasiprimitive} if every
non-trivial normal subgroup of $G$ is transitive. It was observed in 
\cite{a-mcs} that pre-primitivity and quasi\-primitivity together are equivalent
to primitivity. However, this is not the case if we replace pre-primitivity
by the OB property.

For example, the transitive actions of $S_5$ and $A_5$ on $15$ points (on the cosets of a Sylow 2-subgroup) are both
quasi\-primitive but not pre-primitive. However, there is a unique non-trivial
invariant partition in each case, with $5$ parts each of size~$3$; so, by
Corollary~\ref{c:chain}, these groups are OB.

\medskip

Another related concept is that of stratifiability, see~\cite{rab:strat,cap}.
The permutation group $G$ on $\Omega$ is \emph{stratifiable} if the orbits of
$G$ on \emph{unordered} pairs of points of $\Omega$ form an association scheme. 
Since the relations in an association scheme commute, this is equivalent to
saying that the symmetric $G$-invariant relations commute. Since equivalence
relations are symmetric, we conclude:

\begin{prop}
A stratifiable permutation group has the OB property.
\end{prop}

The paper~\cite{cap} defines a related property for a transitive permutation 
group $G$, that of being \emph{AS-friendly}: this holds if there is a unique
finest association scheme which is $G$-invariant. It is easy to see that a
stratifiable group is AS-friendly. So we could ask, is there any relation
between being AS-friendly and having the OB property?

\medskip

In common with many other permutation group properties, the following holds:

\begin{prop}
  The OB property is upward-closed; that is, if $G$ has the OB property
  and $G\le H\le\Sym(\Omega)$ then $H$ has the OB property.
\label{p:up}
\end{prop}

\begin{proof}
The $H$-invariant equivalence relations form a sublattice of the lattice of
$G$-invariant equivalence relations.\qed
\end{proof}

\subsubsection{Products}

We consider direct and wreath products of transitive groups.

\begin{theorem}
Let $G$ and $H$ be transitive permutation groups. Then $G\wr H$ (in its
imprimitive action) has the OB property if and only if $G$ and $H$ do.
\label{p:wr}
\end{theorem}

\begin{proof}
If $G$ and $H$ act on $\Gamma$ and $\Delta$ respectively, then $G\wr H$ acts
on $\Gamma\times\Delta$, and preserves the canonical partition
$\Pi_0$ into the sets $\Gamma_\delta=\{(\gamma,\delta):\gamma\in\Gamma\}$ for
$\delta\in\Delta$. It was shown in \cite{a-mcs} that any invariant partition
for $G\wr H$ is comparable with $\Pi_0$; the partitions below $\Pi_0$ induce
a $G$-invariant partition on each part of $\Pi_0$, while the partitions above
$\Pi_0$ induce an $H$-invariant partition on the set of parts.

Suppose that $G$ and $H$ have the OB property, and let $\Sigma_1$ and $\Sigma_2$
be $G\wr H$-invariant partitions. If one is below $\Pi_0$ and the other above,
then they are comparable, and so they commute. If both are below, then they
commute since $G$ has the OB property; and if both are above, then they
commute since $H$ has the OB property. So the OBS is obtained by nesting the
OBS for $G$ in that for $H$.

Conversely, suppose that $G\wr H$ has the OB property. Then the partitions
below $\Pi_0$ commute, so $G$ has the OB property; and the partitions above
$\Pi_0$ commute, so $H$ has the OB property.\qed
\end{proof}

\begin{cor}
  Let $G$ and $H$ be permutation groups. If $G\times H$ has the OB property
  in its product action then $G$ and $H$ both have the OB property.
\label{c:dp}
\end{cor}

\begin{proof}
As in \cite{a-mcs}, $G\times H$ is a subgroup of $G\wr H$. So, if $G\times H$
has the OB property, then $G\wr H$ has the OB property by
Proposition~\ref{p:up}, and the result holds by Theorem~\ref{p:wr}. \qed
\end{proof}

We will see later (after Theorem~\ref{t:qh}) that the converse is false.
However, we have some positive results.

First we prove some general facts about invariant partitions of direct products of an arbitrary number of groups in their product action, and slightly extend a result in~\cite{a-mcs}, proving that the direct product of an arbitrary number of primitive groups in its product action is pre-primitive. This result is interesting in its own right, but it will also be used to show that a generalised wreath product of primitive groups is pre-primitive, which constitutes a part of Theorem~\ref{thm-gwpPP}. First we give some language to describe partitions of products.

Let $G$ and $H$ act transitively on $\Gamma$ and $\Delta$ respectively, and let $\Pi$ be a $(G\times H)$-invariant partition of $\Gamma\times \Delta$. We define two partitions of $\Gamma$ in the following way:

\begin{itemize}
\item Let $P$ be a part of $\Pi$. Let $P_0$ be the subset of $\Gamma$ defined by 
\[
P_0 = \{\gamma \in \Gamma \, : \, (\exists \delta\in \Delta)((\gamma, \delta)\in P)\}.
\]
We claim that the sets $P_0$ arising in this way are pairwise disjoint. For suppose that $\gamma \in P_0\cap Q_0$, where $Q_0$ is defined similarly for another part $Q$ of $\Pi$; suppose that $(\gamma, \delta_1)\in P$ and $(\gamma, \delta_2)\in Q$. There is an element $h\in H$ mapping $\delta_1$ to $\delta_2$. Then $(1, h)$ maps $(\gamma, \delta_1)$ to $(\gamma, \delta_2)$, and hence maps $P$ to $Q$, and $P_0$ to $Q_0$; but this element acts trivially on $\Gamma$, so $P_0 = Q_0$. It follows that the sets $P_0$ arising in this way form a partition of $\Gamma$, which we call the $G$-\emph{projection partition} of $\Gamma$ induced by $\Pi$.

\item Choose a fixed $\delta\in \Delta$, and consider the intersections of the parts of $\Pi$ with $\Gamma \times \{\delta\}$. These form a partition of $\Gamma\times \{\delta\}$ and so, by ignoring the second factor, we obtain a partition of $\Gamma$ called the $G$-\emph{fibre partition} of $\Gamma$ induced by $\Pi$. Now the action of the group $\{1\}\times H$ shows that it is independent of the element $\delta\in \Delta$ chosen.
\end{itemize}

We note that the $G$-projection partition and the $G$-fibre partition are both $G$-invariant, and the second is a refinement of the first. In a similar way we get $H$-fibre and $H$-projection partitions of $\Delta$, both $H$-invariant.

\begin{prop}
Let $\Pi$ be a $G\times H$-invariant partition of $\Gamma\times\Delta$, where
$G$ and $H$ act transitively on $\Gamma$ and $\Delta$ respectively. Then
the projection and fibre partitions of $\Pi$ on $\Gamma$ are equal if and
only if $\Pi$ is obtained by crossing a $G$-invariant partition of $\Gamma$
with an $H$-invariant partition of $\Delta$.
\end{prop}

\begin{proof}
First we observe that the projection and fibre partitions on $\Gamma$ agree
if and only if those on $\Delta$ agree. For the pairs in a part $P$ of $\Pi$
are the edges of a bipartite graph on $A\cup B$, where $A$ and $B$ are parts of
the projection partitions on $\Gamma$ and $\Delta$ respectively; the valency of a point in $A$ is equal to the number
of points of $B$ in a part of the fibre partition on $\Delta$, which we will
denote by $a$; and similarly the valencies $b$ of the points in $B$. Then
counting edges of the graph (that is, pairs in part $P$ of $\Pi$), we see
that $|A|a=|B|b$. Now the fibre and projection partitions on $\Gamma$ agree if
and only if $|A|=b$, which is equivalent to $|B|=a$.

Moreover, if this equality holds, then every pair in $A\times B$ lies in the
same part of $\Pi$, so $A$ and $B$ are parts of both the projection and
fibre partitions on the relevant sets. In this case, $\Pi$ is obtained by
crossing these partitions.

Conversely, it is easy to see that if $\Pi$ is obtained by crossing, then the
fibre and projection partitions coincide.\qed
\end{proof}

Next we introduce the notion of partition orthogonality.

\paragraph{Definition} Let $G, H$ be
transitive permutation groups, on $\Gamma$, $\Delta$ respectively, as above. We say that $G$ and $H$ are \emph{partition-orthogonal} if the only $G\times H$-invariant partitions of $\Gamma\times\Delta$ are of the form $\{\Gamma_i\times \Delta_j \mid i\in\{1, \ldots, m\}, j\in \{1, \ldots, n\}\}$ where $\{\Gamma_1, \ldots, \Gamma_m\}$ is a $G$-invariant partition of $\Gamma$ and $\{\Delta_1, \ldots, \Delta_n\}$ is an $H$-invariant partition of~$\Delta$.

\begin{lemma}\label{lemma-orthogonal}
Let $G_i \leq \Sym(\Omega_i)$ for $i\in \{1, \ldots, m\}$ be transitive, and let $G = G_1\times \cdots \times G_m$ act on $\Omega = \Omega_1 \times \cdots\times \Omega_m$ component-wise. If $G_i$ and $G_j$ are partition-orthogonal for all $i, j\in \{1, \ldots, m\}$ with $i\ne j$, then the $G$-invariant partitions are precisely the products of $G_i$-invariant partitions for $i\in \{1, \ldots, m\}$.
\end{lemma}

\begin{proof}

We prove the claim by induction. If $m = 2$, then the claim follows by the definition. Suppose that the claim holds for $m - 1$ factors. Let $H = G_1\times \cdots \times G_{m - 1}$ and suppose for a contradiction that there is some $G$-invariant partition $\Pi$ which is not a direct product of partitions of the sets $\Omega_i$. Then the $H$-fibre and $H$-projection partitions induced on $\Omega_1\times \cdots \times \Omega_{m - 1}$ by $\Pi$ must differ. 

By the induction hypothesis, all the $H$-invariant partitions are direct products of partitions, and therefore there must exist some $i\in \{1, \ldots, m - 1\}$ such that  the $G_i$-fibre and the $G_i$-projection partition induced on $\Omega_i$ by $\Pi$ differ. However, this means that the partition induced on $G_i\times G_m$ is not a direct product of partitions of $\Omega_i\times \Omega_m$, which is a contradiction since we have assumed that $G_i$ and $G_m$ are partition-orthogonal.

Therefore, every $G$-invariant partition of $\Omega$ must be a direct product of partitions of the sets $\Omega_i$. \qed
\end{proof}

\begin{lemma}\label{lemma-orthogonal2}
Let $G_1 \leq \Sym(\Omega_1), \ldots, G_m\leq \Sym(\Omega_m), H\leq \Sym(\Delta)$ be transitive groups. If $H$ is partition-orthogonal to $G_i$ for all $i\in \{1, \ldots, m\}$, then $H$ is partition-orthogonal to $G_1\times \cdots \times G_m$.
\end{lemma}

\begin{proof}
Let $G = G_1\times \cdots \times G_m \times H$ and $\Omega = \Omega_1\times \cdots \times \Omega_m \times \Delta$. We prove the claim by induction on $m$.

We first prove the claim for $m = 2$. Let $(\alpha, \beta, \delta), (\alpha', \beta', \delta')\in \Omega_1\times \Omega_2\times \Delta$. First consider the $H$-fibre partition $\Pi_{\Delta}$ of $\Delta$ induced by $\Pi$. Note that by definition $(\alpha, \beta, \delta)$ and $(\alpha, \beta, \delta')$ are in the same part of $\Pi$ if and only if $\delta$ and $\delta'$ are in the same part of $\Pi_{\Delta}$.

Now consider the $(G_1\times H)$-fibre partition $\Pi_{\Omega_1\times \Delta}$ of $\Omega_1\times \Delta$ induced by $\Pi$. Since $G_1$ and $H$ are partition-orthogonal by assumption, we must have $\Pi_{\Omega_1\times \Delta} = \Pi_{\Omega_1}\times \Pi_{\Delta}$, where $\Pi_{\Omega_1}$ denotes the $G_1$-fibre partition of $\Omega_1$ induced by $\Pi$. This means that $(\alpha, \beta, \delta)$ and $(\alpha', \beta, \delta')$ are in the same part of $\Pi$ if and only if $\alpha$ and $\alpha'$ are in the same part of $\Pi_{\Omega_1}$ and $\delta$ and $\delta'$ are in the same part of $\Pi_{\Delta}$.

An entirely similar argument shows that $(\alpha, \beta, \delta)$ and $(\alpha, \beta', \delta')$ are in the same part of $\Pi$ if and only if $\beta$ and $\beta'$ are in the same part of the $G_2$-fibre partition $\Pi_{\Omega_2}$ of $\Omega_2$ induced by $\Pi$ and $\delta$ and $\delta'$ are in the same part of $\Pi_{\Delta}$.

It therefore follows that $(\alpha, \beta,\delta)$ and $(\alpha', \beta', \delta')$ are in the same part of $\Pi$ if and only if $\alpha$ and $\alpha'$ are in the same part of $\Pi_{\Omega_1}$, $\beta$ and $\beta'$ are in the same part of $\Pi_{\Omega_2}$ and $\delta$ and $\delta'$ are in the same part of $\Pi_{\Delta}$, and so
\[
\Pi = \Pi_{\Omega_1}\times \Pi_{\Omega_2}\times \Pi_{\Delta}
\]
which proves the claim.

Now suppose that the claim holds for all integers less than $m$. Then, it follows that $H$ is partition-orthogonal to $G_1\times \cdots \times G_{m - 1}$. Now, since $H$ is partition-orthogonal to both $G_1\times \cdots \times G_{m -1}$ and $G_m$, using the inductive hypothesis once more gives us that $H$ is indeed partition-orthogonal to $G_1\times \cdots \times G_m$. \qed
\end{proof}

\begin{lemma}\label{lemma-orthogonalPP}
Let $G\leq \Sym(\Gamma)$ and $H\leq \Sym(\Delta)$ be partition-orthogonal pre-primitive groups. Then $G\times H$ in its product action is pre-primitive.
\end{lemma}

\begin{proof}
Let $\Pi$ be a $G\times H$-invariant partition of $\Gamma\times \Delta$. Since $G$ and $H$ are partition-orthogonal, $\Pi$ is the direct product of a $G$-invariant partition $\Pi_G$ and an $H$-invariant partition $\Pi_H$. Since both $G$ and $H$ are pre-primitive, it follows that $\Pi_G$ and $\Pi_H$ are orbit partitions of some subgroups $G^*$ and $H^*$ of $G$ and $H$ respectively. It is then easy to check that $\Pi$ is the orbit partition of $G^*\times H^*$, which proves the claim. \qed
\end{proof}

\begin{theorem}\label{thm-directPP}
Let $G_i \leq \Sym(\Omega_i)$ for $i\in \{1, \ldots, m\}$, and let $G_i$ act primitively on $\Omega_i$ for all $i\in \{1, \ldots, m\}$. Then $G = G_1\times \cdots \times G_m$ in its product action is pre-primitive.
\end{theorem}

\begin{proof}

Abelian primitive groups are cyclic of prime order. So, by rearranging the components if necessary, we can write $G$ as a direct product of elementary abelian groups of different prime power order and non-abelian primitive groups. 

It has been shown in~\cite{a-mcs} that two primitive groups are partition-orthogonal if and only if they are not cyclic of the same prime order. Therefore, if $P$ and $Q$ are two elementary abelian groups of orders $p^a$ and $q^b$ respectively, with $p\ne q$, then it follows by Lemma~\ref{lemma-orthogonal2} that every component of $Q$ is partition-orthogonal to $P$, and then applying Lemma~\ref{lemma-orthogonal2} again, we get that $P$ must be partition-orthogonal to $Q$. Similarly, we get that any elementary abelian group and any non-abelian primitive group are partition-orthogonal. Then Lemma~\ref{lemma-orthogonal} gives us that $G$ can be written as  a direct product of mutually partition-orthogonal factors, and it is hence pre-primitive by Lemma~\ref{lemma-orthogonalPP}. \qed
\end{proof}

\subsubsection{Regular groups}\label{ss:regular}

It follows from Corollary~\ref{c:permut} that, if $G$ is a regular permutation
group, then $G$ has the OB property if and only if any two subgroups of $G$ commute. These groups were determined by Iwasawa~\cite{iwasawa}; we refer
to Schmidt~\cite[Chapter 2]{schmidt} for all the material we require. In this
section we use the term \emph{quasi-hamiltonian}, taken from \cite{cd}, for a
group in which any two subgroups commute. (The term will not be used
outside this section.)

We warn the reader that both Iwasawa and Schmidt consider hypotheses which are
more general in two ways:
\begin{itemize}
\item they consider groups whose subgroup lattices are modular, which is weaker
than requiring all subgroups to commute;
\item they consider infinite as well as finite groups.
\end{itemize}
We have not found a reference for precisely what we want, so we give a direct
proof of the first part; the second is \cite[Theorem 2.3.1]{schmidt}.

\begin{theorem}
\begin{enumerate}
\item A finite group $G$ is quasi-hamiltonian if and only if it is the direct
product of quasi-hamiltonian subgroups of prime power order.
\item Suppose that $p$ is prime, and $G$ is a non-abelian quasi-hamiltonian
$p$-group. Then either
\begin{itemize}
\item $G=Q_8\times V$, where $Q_8$ is the quaternion group of order $8$ and
$V$ an elementary abelian $2$-group; or
\item $G$ has an abelian normal subgroup $A$ with cyclic factor group and there is
$b\in G$ with $G=A\langle b\rangle$ and $s$ such that $b^{-1}ab=a^{1+p^s}$ for
all $a\in A$, with $s\ge2$ if $p=2$.
\end{itemize}
\end{enumerate}
\label{t:qh}
\end{theorem}

Here is the proof of part (a). Suppose that $P_1$ and $P_2$ are Sylow
$p$-subgroups of the quasi-hamiltonian group $G$. Then $P_1P_2$ is a subgroup,
and $|P_1P_2|=|P_1|\cdot|P_2|/|P_1\cap P_2|$. Since $P_1$ and $P_2$ are Sylow
subgroups, this implies that $P_1=P_2$. So all Sylow subgroups of $G$ are
normal, and $G$ is nilpotent. Thus it is the direct product of its Sylow
subgroups. Since quasi-hamiltonicity is clearly inherited by subgroups, the
result follows

Conversely, if $G$ is nilpotent with quasi-hamiltonian Sylow subgroups, then
any subgroup is nilpotent and hence a direct product of its Sylow subgroups.
Factors whose orders are powers of different
primes commute; factors whose orders are powers of the same prime commute
by hypothesis. So any two subgroups commute.

\medskip

Note that not every quasi-hamiltonian group is a Dedekind group, namely a group all of whose subgroups are normal; so the OB
property lies strictly between transitivity and pre-primitivity. Note also
that $Q_8$ (acting regularly) is quasi-hamiltonian but $Q_8\times Q_8$ in the product action is not; so the OB property
is not closed under direct product.

\medskip

For groups with a regular normal subgroup, we have the following result.

\begin{theorem}
If $G\leq \Sym(\Omega)$ is a transitive group containing a regular normal subgroup $N$, then $G$ is OB if and only if the subgroups of $N$ normalised by
$G_{\alpha}$ commute.
\end{theorem}

\begin{proof}
Suppose that $G$ is OB. Since $N$ is a regular normal subgroup of $G$ we can write $G = NG_{\alpha}$ for some $\alpha \in \Omega$, where $N\cap G_{\alpha} = 1$ and we can identify $\Omega$ with $N$ in such a way that $G_{\alpha}$ acts by conjugation and $N$ acts by right multiplication. 

We first show that the subgroups containing $G_{\alpha}$ are of the form $HG_{\alpha}$ for some $H\leq N$ invariant under the action of $G_{\alpha}$. Let $F$ be a subgroup containing $G_{\alpha}$. Since $F\leq G = NG_{\alpha}$ all the elements of $F$ are of the form $ng$ where $n\in N$ and $g\in G_{\alpha}$. Then since $G_{\alpha}\leq F$ it follows that $n = ngg^{-1}\in F$. Hence, $F = HG_{\alpha}$, where $H = N\cap F \leq N$. Since both $F$ and $N$ are invariant under the action of $G_{\alpha}$, so is $H$.

Let $KG_{\alpha}$, $LG_{\alpha}$ be two such subgroups. Since $G$ is OB, they
commute (Corollary~\ref{c:permut}), and we have
\begin{equation}\label{eq-ob}
KG_{\alpha}LG_{\alpha} = LG_{\alpha}KG_{\alpha}.
\end{equation}
But $G_{\alpha}L = LG_{\alpha}$ and $G_{\alpha}K = KG_{\alpha}$ since $KG_{\alpha}, LG_{\alpha} \leq G$.
Therefore, by Equation (1) we get
\[
KLG_{\alpha} = LKG_{\alpha}
\]
and intersecting both sides with $N$ gives us $KL = LK$. Since $K, L$ were arbitrary $G_{\alpha}$-invariant subgroups of $N$ the claim holds. 

Conversely, suppose that all the $G_{\alpha}$-invariant subgroups of $N$ commute and consider subgroups $KG_{\alpha}, LG_{\alpha}\leq G$, where $K$ and $L$ are
subgroups of $N$ normalised by $G_\alpha$. Then
\[
KG_{\alpha}LG_{\alpha} = KLG_{\alpha} = LKG_{\alpha} = LG_{\alpha}KG_{\alpha},
\]
and so $G$ is OB (again by Corollary~\ref{c:permut}).
\qed
\end{proof}

\subsubsection{Modularity and distributivity}

We have seen at the start of Section~\ref{ss:regular} that the subgroup lattice of a group, which clearly determines
modularity, does not determine whether the subgroups commute. So we cannot
expect a characterisation of the OB property in terms of the lattice of
subgroups containing a given point stabiliser. But is there anything to say here?

An example of a transitive group in which the lattice of invariant equivalence
relations is the pentagon ($P_5$ in Figure~\ref{f:p5n3}) is the following. Let $G$ be the $2$-dimensional
affine group over a finite field $F$ of order~$q$, and let $G$ act on the set
of \emph{flags} (incident point-line pairs) in the affine plane. The three
non-trivial $G$-invariant relations are ``same line'', ``parallel
lines'', and ``same point''. Clearly the equivalence relations ``same point''
and ``same line'' do not commute.

Since modularity does not imply the OB property, we could ask whether a
stronger property does. We saw in Corollary~\ref{c:chain} that the property
of being a chain does suffice. Is there a weaker property?

\begin{prop}
Let $G$ be a finite regular permutation group. Then the lattice of
$G$-invariant partitions is distributive if and only if $G$ is cyclic.
\end{prop}

This is true because a group with distributive subgroup lattice is locally
cyclic, by Ore's theorem \cite[Section 1.2]{schmidt}, and a finite locally
cyclic group is cyclic. Since a cyclic group is Dedekind, it is pre-primitive
and so has the OB property.

However, there is no general result along these lines. Even if we assume that
the lattice of $G$-invariant partitions is a \emph{Boolean lattice}
(isomorphic to the lattice of subsets of a finite set), the group may fail to
have the OB property, as the next example shows.

\paragraph{Example} Let $G=\mathrm{GL}(n,q)$ acting on the set of maximal
chains of non-trivial proper subspaces
\[V_1<V_2<\cdots<V_{n-1}\]
in the vector space $V=\mathrm{GF}(q)^n$, where $\dim(V_k)=k$ for $0<k<n$.
The stabiliser $B$ of such a chain is a \emph{Borel subgroup} of $G$; if we 
take $V_k$ to be spanned by the first $k$ basis vectors, then $B$ is the group
of upper triangular matrices with non-zero entries on the diagonal. From the
theory of algebraic groups, it is known that the only subgroups of $G$
containing $B$ are the \emph{parabolic subgroups}, the stabilisers of subsets
of $\{V_1,\ldots,V_{n-1}\}$ (see for example \cite{humphreys} for the theory).
Hence the lattice of $G$-invariant partitions is isomorphic to the Boolean
lattice $B_{n-1}$ of subsets of $\{1,\ldots,n-1\}$ (the isomorphism reverses
the order since the stabiliser of a smaller set of subspaces is larger).

However, the equivalence relations do not all commute. Consider the relations
$\Pi_1$ and $\Pi_2$ corresponding to the subgroups fixing $V_1$ and $V_2$.
Thus, two chains are in the relation $\Pi_1$ if they contain the same
$1$-dimensional subspace, and similarly for $\Pi_2$. Now starting from the
chain $(V_1,V_2,\ldots,V_n)$, a move in a part of $\Pi_2$ fixes $V_2$ and
moves $V_1$ to a subspace $V_1'$ of $V_2$; then a move in $\Pi_1$ fixes
$V_1'$, so the resulting chain begins with a subspace of $V_2$. But if
we move in a part of $\Pi_1$, we can shift $V_2$ to a different $2$-dimensional
subspace, and then a move in a part of $\Pi_2$ can take $V_1$ to a subspace
not contained in $V_2$. So $\Pi_1\circ\Pi_2\ne\Pi_2\circ\Pi_1$, and the
lattice is not an OBS.

So $G$ does not have the OB property, even though the lattice of $G$-invariant
partitions is a Boolean lattice (and hence distributive).

\subsection{Generalised wreath products}

In this section, we prove two main results. The first describes the
group-theoretic structure of a generalised wreath product, and will be needed
later. The second investigates properties of the generalised wreath product
of primitive groups; in particular, they are pre-primitive
and hence have the OB property, and we give necessary and sufficient
conditions for them to have the PB property. 

\subsubsection{A group-theoretic result}

First we prove a result about generalised wreath products which will be needed
later.

We note that, if $p$ is a minimal element of a poset $M$, then $\{p\}$ is a
down-set, and so corresponds to a partition $\Pi$ of the domain $\Omega$ of
the generalised wreath product of a family of groups over $M$.

\begin{theorem}
Let $G$ be the generalised wreath product of the groups $G(m)$ over a
poset~$M$, acting on a set $\Omega$. Let $p$ be a minimal element of $M$. Let
$\Pi$ be the corresponding partition of $\Omega$, $H$ the group induced on the
set of parts by $G$, $N$ the stabiliser of all parts of $\Pi$. Then
\begin{enumerate}
\item $H$ is isomorphic to the generalised wreath product of the groups $G(q)$
for $q\in M\setminus\{p\}$;
\item $N$ is a direct product of copies of $G(p)$, where there is an
equivalence relation $\sim$ on the set of parts of $\Pi$ (determined by the poset $M$)
such that each direct factor acts in the same way on the parts in one
equivalence class and fixes every point in the other parts;
\item $G$ is a semidirect product $N\rtimes H$.
\end{enumerate}
\label{t:gwp_sdp}
\end{theorem}

\begin{proof}
For (a), we note that, since $p$ is minimal, suppressing the
$p$th coordinate of every tuple in $\Omega$ gives the generalised wreath
product of the remaining groups indexed by the elements different from $p$.

\medskip

Part (b) is proved using the definition of a generalised wreath product.
The equivalence relation is defined as follows: for parts $P$ and $Q$
of $\Pi$, $P\sim Q$
if and only if $P$ and $Q$ lie in the same part of 
$\Pi\vee \Phi$ for all partitions $\Phi$ of the poset block structure
defined by $M$ which are incomparable to~$\Pi$.

First note that since $N$ fixes the parts of $\Pi$, it must also fix the parts
of every partition lying above $\Pi$. Therefore, only parts of partitions
incomparable to $\Pi$ can be moved by $N$. Now let $\Phi$ denote a
partition incomparable to $\Pi$. Note that since
$\Pi\preccurlyeq \Pi\vee \Phi$, the parts of $\Phi$ contained in the same
part of $\Pi\vee \Phi$ can only be permuted amongst themselves by $N$. Hence,
if the actions of $h\in N$ on two parts of $\Pi$ are equivalent, then those two parts must be contained in the same part of
$\Pi\vee \Phi$ for every partition $\Phi$ of $\Omega$ incomparable to $\Pi$.

It now remains to show that if $P,Q\in \Pi$ are such that
$P\sim Q$, then $N$ acts in the same way on $P$ and $Q$.
Let $\gamma, \delta$ lie in $P$ and $Q$, and moreover suppose that they
are contained in the same part of $\Phi$ for every partition $\Phi$ of $\Omega$
incomparable to $\Pi$. It suffices to show that every $h\in N$ maps $\gamma$
and $\delta$ to the same part of $\Phi$ for every $\Phi$ incomparable
to~$\Pi$.

Now $h$ can be written as a product $\prod_{\Phi} h_{\Phi}$, where each factor $h_{\Phi}$ encodes the permutation induced by $h$ of the parts of the corresponding partition $\Phi$ of $\Omega$ induced by $h$. Hence, it suffices to show that $h_{\Phi}$ maps $\gamma$ and $\delta$ to the same part for an arbitrary partition $\Phi$ of $\Omega$ incomparable to $\Pi$. 
We may assume without loss of generality that $\Phi$ is join-indecomposable,
since every element is a join of JI elements, and the distributive law implies
that if a collection of JI elements are incomparable with $\Pi$ then so is
their join.

Let $m$ be the element corresponding to $\Phi$ in the poset $M$. Using the notation established in~\cite{bprs}, we note that $\gamma$ and $\delta$ must be such that $\gamma_i = \delta_i$ for all $i\sqsupseteq m$ in $M$. Therefore,
\[
(\gamma h_{\Phi})_i = \gamma_i(\gamma \pi^i (h_{\Phi})_i) = \delta_i(\delta \pi^i (h_{\Phi})_i) = (\delta h_{\Phi})_i
\]
for all $i\sqsupseteq m$, which proves the claim.

We finally note that $\sim$ is dependent only on the poset $M$ and not the group $G$.

\medskip

For (c), we have to show that $H$ normalises $N$ and that the action of $H$
extends to $\Omega$. The first statement is clear since $N$ is the subgroup
fixing all parts of $\Pi$. For the second, note that $H$ acts on the set
of $(|M|-1)$-tuples; extend each element to act on $|M|$-tuples by acting
as the identity on the $p$-th coordinate.\qed
\end{proof}

\subsubsection{Generalised wreath products of primitive groups}

In this section, we will use the notation for poset block structures and generalised wreath products defined in Section~\ref{s:pbs}. Moreover, let $[N]$ denote the set $\{1, \ldots, N\}$, and for every subset $J$ of $M$, let $X_J$ be the index set of $J$, namely $\{i\in [N] \, : \, m_i\in J\}$. We then define $P_J$ to be the partition of $\Omega$ whose set of parts is 
\[
\left\{\prod_{j \in X_J} \Omega_j \times \prod_{k\in [N]\setminus X_J} \{\alpha_k\} : \alpha_k \in \Omega_k \hbox{ for all } k\in [N]\setminus X_J\right\}.
\]

We now prove a small lemma that will be used in the proof of Theorem~\ref{thm-gwpPP}.

\begin{lemma}
Let $G$ be the generalised wreath product of the groups $G(m_i)$ over the
poset $M$.
Let $J, K$ be down-sets of $M$ such that $P_K \preccurlyeq P_J$, let $\Gamma$ be a part of $P_{J}$, and let $\Delta$ be the set of parts of $P_K$ contained in $\Gamma$. Then the permutation group $G(\Delta,\Gamma)$ induced by the setwise stabiliser of $\Gamma$ on $\Delta$ is isomorphic to the generalised wreath product of the groups $G(m_i)$ for $i\in X_J\setminus X_K$.
\end{lemma}

\begin{proof}
Let $\Gamma = \prod_{j \in X_J} \Omega_j \times \prod_{i\in [N]\setminus X_J} \{\alpha_i\}$, where $\alpha_i$ is fixed for $i\in [N]\setminus X_J$. Note that the setwise stabiliser $G_\Gamma$ inside $G$ must be equal to the generalised wreath product of the groups $H(m_i)$, where $H(m_i) = G(m_i)$ for all $i\in X_J$ and $H(m_i) = (G(m_i))_{\alpha_i}$ for $i\not\in X_J$. Now since the elements of $\Delta$ are blocks of imprimitivity of $G$, they are also blocks of $G_\Gamma$, and moreover, since $\Gamma$ is a block of $P_J$, it follows that $G_\Gamma$ induces a permutation group on $\Delta$. Let $\rho$ denote the associated permutation representation.

Note that every element of $\Delta$ is of the form $\prod_{i\in X_K}\Omega_i \times \prod_{i\in [N]\setminus X_K}\{\alpha_i\}$, where $\alpha_i$ is fixed for $i\in X_J\setminus X_K$. Therefore, $\ker\rho$ must fix all elements of $\Omega_i$ for $i\in X_J\setminus X_K$, must fix $\alpha_i$ for $i\in [N]\setminus X_J$, and can permute the elements of $\Omega_i$ for $i\in X_K$ in any way $G_\Gamma$ allows. Hence, $\ker \rho$ is equal to the generalised wreath product of $L(m_i)$, where $L(m_i) = G(m_i)$ for $i\in X_K$, $L(m_i) = 1$ for $i\in X_J\setminus X_K$, and $L(m_i) = G(m_i)_{\alpha_i}$ for $i\in [N]\setminus X_J$. We then deduce that 
\[
G(\Delta,\Gamma) \cong G_\Gamma/ \ker \rho, \]
the generalised wreath product of $G(m_i)$ for $i\in X_J\setminus X_K$,
as claimed. \qed
\end{proof}

We are now in a position to state and prove the main theorem of this section.

\begin{theorem}\label{thm-gwpPP}
If $G(m_i)$ is primitive for every $i\in [N]$, then the following hold for
their generalised wreath product $G$:
\begin{enumerate}
\item $G$ is pre-primitive, and hence has the OB property;
\item the following are equivalent:
  \begin{enumerate}
  \item $G$ has the PB property;
  \item the only $G$-invariant partitions are the ones corresponding to
down-sets in $M$;
  \item there do not exist incomparable elements $m_i,m_j\in M$ such that
$G(m_i)$ and $G(m_j)$ are cyclic groups of the same prime order.
  \end{enumerate}
\end{enumerate}
\end{theorem}

\begin{proof}
Let $K$ denote the direct product of the groups $G(m_i)$ (for $i\in[N]$)
in its product action and $P$ denote the lattice of partitions corresponding to down-sets in $M$.

(a) Since pre-primitivity is upward-closed, it suffices to show that $K$ can be embedded in $G$. Then the claim will follow by Theorem~\ref{thm-directPP}. Let $H$ be the set of all functions $f = \prod_{i\in [N]} f_i \in G$ such that $f_i$ sends all elements of $\Omega^i$ to the same element of $G(m_i)$ for all $i\in [N]$. We will show that $H$ is permutation isomorphic to~$K$.
 
 We first start by showing that $H\leq G$. To prove closure, it suffices to show that for $f, h\in H$, then $(fh)_i$ sends all elements of $\Omega^i$ to the same element of $G(m_i)$ for every $i\in [N]$. We can do this by showing that if $\gamma, \delta\in \Omega$, then $fh$ acts on both $\gamma$ and $\delta$ with the same group element on each coordinate. We will slightly abuse notation and for $f = \prod_{i\in [N]} f_i\in H$, we will write $\im(f_i)$ for the element that $f_i$ maps all the elements of $\Omega^i$ to, instead of the set containing just this element. Now let $g = \im(f_i)$ and $g' = \im(h_i)$, then
 \[
 (\gamma fh)_i = (\gamma f)_i(\gamma f\pi^ih_i) = \gamma_i(\gamma\pi^if_i)(\gamma f\pi^ih_i) = \gamma_igg' 
 \]
 for all $i\in [N]$. Similarly,
  \[
 (\delta fh)_i = (\delta f)_i(\delta f\pi^ih_i) = \delta_i(\delta\pi^if_i)(\delta f\pi^ih_i) = \delta_igg' 
 \]
for all $i\in [N]$, and therefore, $fh\in H$. For $i$ in $[N]$, let $z_i$ be the function which maps all the elements of $\Omega^i$ to the inverse of ${\rm im}(f_i)$. Put $z = \prod_{i\in [N]}z_i$. Then
\[
(\gamma fz)_i = (\gamma f)_i(\gamma f\pi^iz_i) = \gamma_i(\gamma\pi^if_i)(\gamma f\pi^iz_i) = \gamma_igg^{-1} = \gamma_i
\]
for all $i\in [N]$, and thus $z = f^{-1}\in H$.

We now need to show that $H$ is permutation isomorphic to $K$ in its product action on $\Omega$. Let $\phi: G \to K$ be the function defined by the formula $(\prod_{i\in [N]}f_i) \phi = \prod_{i\in [N]} \im(f_i)$, and let $\id$ denote the identity function. Note that $\phi$ is clearly a bijection by construction, and also
\[
(\delta f)_i = \delta_ig_i
\]
for all $i\in [N]$, where $g_i=\im(f_i)$, and hence
\[(\delta f) = \delta \left(\prod_{i\in[N]} \im(f_i)\right) = \delta (f \phi) = (\delta \id)(f\phi) ,\] 
 which completes the proof of (a). 
 
\medskip

We now prove (b).  

\vspace{3mm}

\noindent (ii) $\Rightarrow$ (i). Note that if the only partitions preserved by $G$ are the ones in $P$, then clearly $G$ is PB. 

\vspace{3mm}

\noindent (iii) $\Rightarrow$ (ii). We prove the contrapositive. Suppose that there are further partitions fixed by $G$ other than the ones corresponding to down-sets in $M$, and let $\Pi$ be such a partition. Since $K\leq G$, it follows that $\Pi$ is also preserved by $K$. Therefore, by the arguments in the proof of Theorem~\ref{thm-directPP} we deduce that since all of the $G(m_i)$s are primitive, one of the following must hold:
\begin{itemize}
\item All the $G(m_i)$s are mutually partition-orthogonal, and so $\Pi$ is of the form

\[
(\alpha_1, \alpha_2, \ldots, \alpha_n) \sim_J (\beta_1, \beta_2, \ldots, \beta_n) \iff (\forall i\not\in X_J)(\alpha_i = \beta_i),
\]

where $J$ is not a down-set of $M$;
\item at least two of the $G(m_i)$s, say $G(m_i)$ and $G(m_j)$, are cyclic of the same prime order, and $\Pi$ is a partition whose corresponding $G(m_i)$ and $G(m_j)$-fibre partitions are different to the $G(m_i)$ and $G(m_j)$-projection partitions respectively. Since the degree of $G(m_i)$ and $G(m_j)$ is prime, this can only happen if the fibre partitions are the partitions into singletons and the projection partitions are the partitions into a single part.
\end{itemize}

If $\Pi$ is of the first type, then there exist some $i, j\in [N]$ such that $m_i \sqsubset m_j$ and $m_j\in J$, but $m_i\not\in J$. Since $m_i$ and $m_j$ are comparable, there exists a chain $(m_i = a_0, a_1, \ldots, a_k = m_j)$ in $M$. Thus, $\Pi$ must be preserved by the wreath product $G(a_0)\wr G(a_1)\wr \ldots \wr G(a_k)$. However, we know that an imprimitive iterated wreath product cannot preserve partitions of equivalence relations where
\[
(\alpha_{a_0}, \ldots, \alpha_{a_k}) \sim (\beta_{a_0}, \ldots, \beta_{a_k}),
\] 
with $\alpha_{a_s} = \beta_{a_s}$ but $\alpha_{a_t} \neq \beta_{a_t}$ for some $s, t$ such that $s < t$, because for every $l\in \{1, \ldots, k\}$, the group $G(a_l)$ permutes whole copies of $\Omega_r$ for each $r \in \{0, \ldots, l-1\}$. 
 
 Hence, $\Pi$ must be of the second type and thus there exist $G(m_i)$ and $G(m_j)$ cyclic of the same prime order and $\Pi$ is a partition whose corresponding $G(m_i)$ and $G(m_j)$-fibre partitions are the partitions into singletons and the $G(m_i)$ and $G(m_j)$-projection partitions are the partitions into a single part. If $m_i$ and $m_j$ are related, say $m_i \sqsubset m_j$ then, as above, $\Pi$ must be preserved by the iterated wreath product  $G(a_0)\wr G(a_1)\wr \ldots \wr G(a_k)$. However, knowing what partitions imprimitive iterated wreath products preserve, we deduce that $G(a_0)\wr G(a_1)\wr \ldots \wr G(a_k)$ cannot preserve $\Pi$ and therefore $m_i$ and $m_j$ must be incomparable.
 
 \vspace{3mm}

\noindent (i) $\Rightarrow$ (iii). We again prove the contrapositive. So suppose that there are two incomparable
elements in $M$, say $m_1$ and $m_2$, such that the corresponding groups are
cyclic of the same prime order $p$. As defined in Section~\ref{s:pbs},
\[D(i)=\{m\in M:m\sqsubset m_i\}\]
for $i=1,2$. Set
\[S=D(1)\cup D(2),\quad Q=S\setminus\{m_1\}, \quad
R=S\setminus\{m_2\},\hbox{ and } T=Q\cap R.\]
These four sets are all down-sets, and the interval between $T$ and $S$ has the
group $G(m_1)\times G(m_2)$ acting, and so we can find partitions fixed by
the group, other than the ones corresponding to down-sets in $M$. More precisely, there are $p + 1$ partitions corresponding to orbit partitions of the diagonal subgroups of $G(m_1)\times G(m_2)$, and thus preserved by $G(m_1) \times G(m_2)$. If $Y$ is one of those, then $S, T, Q, R, Y$ form a $N_3$ sublattice (Figure~\ref{f:p5n3})
of the invariant partition lattice of $G$, and hence $G$ fails the PB property. This proves the claim. \qed
\end{proof}

\subsection{The embedding theorem}

The Krasner--Kaloujnine theorem~\cite{kk} says that, if $G$ is a transitive but
imprimitive permutation group on $\Omega$, then $G$ is embeddable in the wreath product
of two groups which can be extracted from $G$ (the stabiliser of a block acting
on the block, and $G$ acting on the set of blocks).

In this section, we extend this result to transitive groups which preserve a
poset block structure (a distributive lattice of commuting equivalence
relations). In particular, our result holds for groups with the PB property.
As explained in Subsection~\ref{s:pbs},
such a lattice $\Lambda$ is associated with a poset $M$ (so that $M$ consists
of the non-$E$ join-indecomposable elements of $\Lambda$, and $\Lambda$ consists
of the down-sets in $M$). We want to associate a group with each element
$m\in M$ such that $G$ is embedded in the generalised wreath product of these
groups over the poset $M$.

Our first attempt was as follows. Take $m\in M$; it corresponds to a
join-indecomposable partition $\Pi\in\Lambda$. The join-indecomposability of
$\Pi$ implies that there is a unique partition $\Pi^-$ in $\Lambda$
which is maximal with respect to being below $\Pi$. Then let $G(m)$ be
the permutation group induced by the stabiliser of a part of $\Pi$ acting
on the set of parts of $\Pi^-$ it contains.

However, this does not work. Take $G$ to be the symmetric group $S_6$. This
group has an outer automorphism, and so has two different actions on sets of
size $6$. Take $\Omega$ to be the Cartesian product of these two sets. The
invariant partitions for $G$ are $E$ and $U$ together with the rows $R$ and
columns $C$ of the square. Then $\{E,R,C,U\}$ is a poset block structure.
Both $R$ and $C$ are join-indecomposable, and
$R^-=C^-=E$. Thus $M$ is a $2$-element antichain $\{r,c\}$, and $G(r)$ is
the stabiliser of a row acting on the points of the row, which is the
group $\mathrm{PGL}(2,5)$, and similarly $G(c)$. However, $S_6$ is clearly
not embeddable in $\mathrm{PGL}(2,5)\times\mathrm{PGL}(2,5)$.

So we use a more complicated construction. Given $\Pi$ and $\Pi^-$ as 
above, where $\Pi$ corresponds to $m\in M$, 
let $\mathcal{G}(m)$ be the set of partitions $\Phi\in\Lambda$ satisfying
$\Phi\wedge\Pi=\Pi^-$. For $\Phi\in\mathcal{G}(m)$, let $G_\Phi(m)$
be the group induced on the set of parts of $\Phi$ contained in a given
part of $\Phi\vee\Pi$.

\begin{lemma}
\begin{enumerate}
\item $\mathcal{G}(m)$ is closed under join.
\item If $\Phi_1,\Phi_2\in\mathcal{G}(m)$ with $\Phi_1\preccurlyeq\Phi_2$,
then there is a canonical embedding of $G_{\Phi_1}(m)$ into
$G_{\Phi_2}(m)$.
\end{enumerate}
\label{l:mono}
\end{lemma}

\begin{proof}
The first part is immediate from the distributive law: if 
$\Phi_1,\Phi_2\in\mathcal{G}(m)$, then
\[(\Phi_1\vee\Phi_2)\wedge\Pi=(\Phi_1\wedge\Pi)\vee(\Phi_2\wedge\Pi)
=\Pi^-\vee\Pi^-=\Pi^-.\]

For the second part, we use the fact that, for a given point $\alpha\in\Omega$,
there is a natural correspondence between partitions and certain subgroups of
$G$ containing $G_\alpha$, where the partition $\Pi$ corresponds to the setwise
stabiliser of the part of $\Pi$ containing $\alpha$; meet and join correspond
to intersection and
product of subgroups. Let $H_1$, $H_2$, $P$, $P^-$ be the subgroups
corresponding to $\Phi_1$, $\Phi_2$, $\Pi$, $\Pi^-$. Then the definition
of $\mathcal{G}(m)$ shows that $H_i\cap P=P^-$ for $i=1,2$, while the partitions
$\Phi_i\vee\Pi$ correspond to the subgroups $H_iP$. The actions we are
interested in are thus $H_iP$ on the cosets of $H_i$. We have
\[|H_iP:H_i|=|P:H_i\cap P|=|P:P^-|\]
for $i=1,2$; so coset representatives of $P^-$ in $P$ are also coset
representatives for $H_i$ in $H_iP$. Thus we have a natural correspondence
between these sets. Since $H_1\le H_2$, we have $H_1P\le H_2P$, and the
result holds.\qed
\end{proof}

Hence if $\Psi$ is the (unique) maximal element of $\mathcal{G}(m)$, then
the group $G_\Psi(m)$, which we will denote by $G^*(m)$, embeds all
the groups $G_\Phi(m)$ for $\Phi\in\mathcal{G}(m)$.

Now we can state the embedding theorem.

\begin{theorem}
Let $G$ be a transitive permutation group which preserves a poset block
structure $\Lambda$, and let $M$ be the associated poset.
Define the groups $G^*(m)$ for $m\in M$ as above. Then $G$ is embedded in the
generalised wreath product of the groups $G^*(m)$ over $m\in M$.
\label{t:embed}
\end{theorem}

We remark that this theorem generalises the theorem of Krasner and Kaloujnine.
If $\Pi$ is a non-trivial $G$-invariant partition, then $\{E,\Pi,U\}$ is a
poset block structure; the corresponding poset is $M=\{m_1,m_2\}$, with
$m_1$ and $m_2$ corresponding to the partitions $\Pi$ and $U$; this
$G^*(m_1)$ is the group induced by the stabiliser of a part of $\Pi$ on its
points, and $G^*(m_2)$ the group induced by $G$ on the parts of $\Pi$, as
required.

The proof uses properties of distributive lattices: we deal with some of
these first. Since these lemmas are not specifically about lattices of
partitions, we depart from our usual convention and use lower-case italic
letters for elements of a lattice, and $0$ and $1$ for the least and
greatest elements respectively.

\begin{lemma}\label{lem:cancellation}
Let $L$ be a distributive lattice. If $a,x,y\in L$ satisfy
\[a\wedge x=a\wedge y\hbox{ and }a\vee x=a\vee y,\]
then $x=y$.
\label{l:eq}
\end{lemma}

\begin{proof}
Suppose first that $x\le y$. Then
\begin{eqnarray*}
y &=& y\vee(a\wedge x) \\
&=& (y\vee a)\wedge(y\vee x)  \\
&=& (x\vee a)\wedge(x\vee y)  \\
&=& x\vee(a\wedge y) \\
&=& x\vee(a\wedge x) \\
&=& x.
\end{eqnarray*}

Now let $x$ and $y$ be arbitrary, and put $z=x\wedge y$. Then $z\le x$ and
\begin{eqnarray*}
a\wedge z &=& (a\wedge x)\wedge(a\wedge y)=a\wedge x\\
a\vee z &=& (a\vee x)\wedge (a\vee y)=a\vee x.
\end{eqnarray*}
By the first part, $z=x$. Similarly $z=y$, so $x=y$.\qed
\end{proof}

\noindent \textbf{Remark.} The identity in Lemma~\ref{lem:cancellation} is commonly known as the \emph{cancellation property for distributive lattices}, which appears in~\cite{dp} as Exercise 6.6. Note that distributivity and the cancellation property are in fact equivalent lattice properties, but we only use one direction here, so we only prove the direction we use. The converse can be proved using Theorem~\ref{thm:diamond-pentagon}. In particular, it is clear that $P_5$ and $N_3$ do not admit cancellation.

\begin{lemma}
Suppose that $L$ is the lattice of down-sets in a poset $M$. Let $p$ be
a minimal element of $M$ (so that $\{p\}$ is a down-set). Then the
interval $[\{p\},1]$ in $L$ is isomorphic to the lattice of down-sets
in $M\setminus\{p\}$.
\label{l:ji}
\end{lemma}

\begin{proof}
Let $z=\{p\}$. Let $\JI(L)$ be the set of join-indecomposables in $L$. Since lattices are generated by their join-indecomposable elements, it suffices to
construct an order-isomorphism $F$ from $\JI(L)\setminus z$ to
$\JI([z,1])$. 
The map $F$ is defined by
\[F(a)=a\vee z\]
for $a\in\JI(L)\setminus\{z\}$. We have to show that it is a bijection and
preserves order. First we show that its image is contained in $\JI([z,1])$.

Take $a\in\JI(L)$, $a\ne z$. If $z\le a$, then $a\vee z=a$ and this is
join-indecomposable in $[z,1]$. Suppose that $z\not\le a$. If $a\vee z$
is not JI in $[z,1]$, then there exist $b,c\in[z,1]$ with $b,c\ne a\vee z$ and
$b\vee c=a\vee z$. Then
\[(a\wedge b)\vee(a\wedge c)=a\wedge(b\vee c)=a\wedge(a\vee z)=a.\]
 Since
$a$ is join-indecomposable, we have, without loss of generality, $a\wedge  b=a$, so $a\le b$. Since we also have $z\le b$, it follows that $a\vee z \le b$, and so $a\vee z = b$, a contradiction.

We show that the map is onto. Let $a\in\JI([z,1])$. If $a\in\JI(L)$ then
$a=F(a)$; so suppose not. Then $a=b\vee c$ for some $b,c\in L$. Then
\[z=a\wedge z=(b\vee c)\wedge z=(b\wedge z)\vee(c\wedge z),\]
so at least one of $b$ and $c$ (but not both) is in $[z,1]$, say $b\in[z,1]$.
Then $a=b\vee(c\vee z)$. Since $a\in\JI([z,1])$ and $a\ne b$, we must have
$c\vee z=a$. We claim that $c$ is join-indecomposable. For if $c=d\vee e$, then
\[a=c\vee z=(d\vee z)\vee(e\vee z).\]
If $d\vee z=a=c\vee z$, then $c=d$ (since $d\wedge z=0=c\wedge z$), a
contradiction. The other case leads to a similar contradiction.

Next we show that $F$ is one-to-one. Suppose that $F(a_1)=F(a_2)$. If
$a_1,a_2\in[z,1]$, then $a_1=a_2$. If $a_1,a_2\notin[z,1]$, then
$a_1\vee z=F(a_1)=F(a_2)=a_2\vee z$; also $a_1\wedge z=0=a_2\wedge z$. By
Lemma~\ref{l:eq}, $a_1=a_2$. So suppose that $a_1\in[z,1]$, $a_2\notin[z,1]$.
Then $a_1=F(a_1)=F(a_2)=a_2\vee z$, contradicting the fact that $a_1$ is join-indecomposable.

Finally we show that $F$ is order-preserving. Suppose that $a_1\le a_2$.
If $a_1,a_2\in[z,1]$, then $F(a_1)=a_1\le a_2=F(a_2)$. If $a_1,a_2\notin[z,1]$,
then $F(a_1)=a_1\vee z\le a_2\vee z=F(a_2)$. We cannot have $a_1\in[z,1]$
and $a_2\notin[z,1]$, since then $z\le a_1\le a_2$ but $z\not\le a_2$. Finally
suppose that $a_1\notin[z,1]$ but $a_2\in[z,1]$, so that $z\le a_2$ and
$a_1\le a_2$, then $F(a_1)=a_1\vee z\le a_2=F(a_2)$.\qed
\end{proof}

Now we turn to the proof of Theorem~\ref{t:embed}. The proof is by induction on
the number of elements in $M$. We take
$\Pi_0$ to be a minimal non-$E$ partition, corresponding to a minimal element
$p\in M$. We decorate things computed in the interval $[\Pi_0,U]$ with bars;
for example, $\bar G^*(q)$ corresponds to the group associated in this lattice
with the element $q\ne p$ (which is not in general the same as $G^*(q)$). Thus
$\bar G$ is the group induced by $G$ on the set of parts of $\Pi_0$, which
is a PB group with associated poset $M\setminus\{p\}$; our induction
hypothesis will imply that the group $\bar G$ is embedded in the generalised
wreath product of the groups $\bar G^*(q)$ for $q\in M\setminus\{p\}$.

Let $\Pi$ be a join-indecomposable partition in $[\Pi_0,U]$, corresponding to the element
$q\in M\setminus\{p\}$. As we saw in the proof of Lemma~\ref{l:ji},
there are two possibilities:
\begin{itemize}
\item \textit{Case 1:} $\Pi$ is join-indecomposable in the lattice $L$ of downsets of $M$. Then $\Pi^-$
is above $\Pi_0$, and so the group $\bar G^*(q)$ is the same as $G^*(q)$.
\item \textit{Case 2:} $\Pi=\Pi_0\vee\Psi$, where $\Psi$ is join-indecomposable in $L$
and $\Pi_0$ is not below $\Psi$. Consider the set $\bar{\mathcal{G}}(q)$, where
the bar denotes that it is computed in the lattice $[\Pi_0,U]$. A partition
$\Phi$ belongs to this set if it is above $\Pi_0$ and satisfies
$\Phi\wedge\Pi=\bar\Pi^{-}$, where again the bar denotes the unique maximal
element below $\Pi$ in $[\Pi_0,U]$. An easy exercise shows that
$\bar\Pi^{-}\wedge\Psi=\Psi^-$; hence
\[\Phi\wedge\Psi=\Psi^-,\]
and so $\Phi$ belongs to $\mathcal{G}(q)$. In other words, we have shown that
\[\bar{\mathcal{G}}(q)\subseteq\mathcal{G}(q).\]
By Lemma~\ref{l:mono}, $\bar G^*(q)$ is canonically embedded in $G^*(q)$.
\end{itemize}
In other words, $\bar G^*(q)\le G^*(q)$ for all $q\in M\setminus\{p\}$.
Now, using the induction hypothesis, the group $\bar G$ induced by $G$ on the
parts of $\Pi_0$ is embedded in the generalised wreath product of the
groups $G^*(q)$ over $q\in M\setminus\{p\}$.

Next, consider the normal subgroup $N_0$ of $G$ which fixes every part of
$\Pi_0$. Because $G$ preserves the poset block structure, $N_0$ is contained
in the automorphsm group of this structure, which is a generalised wreath
product of symmetric groups, by Proposition~\ref{p:pbs}. Hence there is
an equivalence relation on the set of parts of $\Pi_0$ as described in
Theorem~\ref{t:gwp_sdp}; the subgroup of the generalised wreath product
fixing all parts of $\Pi_0$
is a direct product of symmetric groups. Since the stabiliser in $G$
of a part of $\Pi_0$ induces the group $G(p)$ on it, we see that $N_0$ is
actually contained in the direct product of copies of $G(p)$, where the
conditions of Theorem~\ref{t:gwp_sdp} apply to this product. Since
$G(p)\le G^*(p)$, we have that $N_0$ is contained in the stabiliser of the
parts of $\Pi_0$ in the generalised wreath product of the groups $G^*(q)$.
We call this stabiliser $N^*$.

In Theorem~\ref{t:gwp_sdp}, we saw that the generalised wreath product $G^*$ of the groups $G^*(q)$
is the semidirect product $N^*\rtimes H^*$ of this normal subgroup by the generalised wreath product
$H^*$ of the groups $G^*(q)$ for $q\ne p$. Now $G$ has a normal subgroup which
is contained in $N^*$, and a complement which is contained in $H^*$; so $G$
is contained in $G^*$. This completes the proof of Theorem~\ref{t:embed}.\qed

\subsection{Intersections of posets}

If $G_1$ and $G_2$ are permutation groups on $\Omega_1$ and $\Omega_2$
respectively, then $G_1\times G_2$ is a subgroup of $G_1\wr G_2$; indeed,
$G_1\times G_2$ is the intersection of $G_1\wr G_2$ and $G_2\wr G_1$. We
are going to extend this to arbitrary generalised wreath products.

Given a family $(G(i) \leq \Sym(\Omega_i) \, : \, m_i\in M)$ of transitive permutation groups indexed by a set $M$, any partial
order on $M$ gives rise to a generalised wreath product of the groups. So we
have a map from partial orders on $M$ to generalised wreath
products of the groups $G(i)$. In this section, we prove that this
map preserves order and intersections. To explain the terminology, inclusions
and intersections of partial orders on the same sets are given by inclusions
and intersections of the sets of ordered pairs comprising the order relations.
It is easy to show that the intersection of partial orders is a partial order.

\begin{theorem}
Let $(G(i) \leq \Sym(\Omega_i) \, : \, m_i\in M)$ be a family of transitive permutation groups indexed by a set $M$, and let
$\mathcal{M}_1=(M,\sqsubseteq_1)$ and $\mathcal{M}_2=(M,\sqsubseteq_2)$ be two
posets based on $M$. Then
\begin{enumerate}
\item the intersection of the generalised wreath products of the groups over
$\mathcal{M}_1$ and $\mathcal{M}_2$ is the generalised wreath product
over the intersection of $\mathcal{M}_1$ and $\mathcal{M}_2$;
\item if $\mathcal{M}_1$ is included in $\mathcal{M}_2$, then the
generalised wreath product over $\mathcal{M}_1$ is a subgroup of the
generalised wreath product over $\mathcal{M}_2$.
\end{enumerate}
\label{t:poset2gwp}
\end{theorem}

\paragraph{Proof} (a) We first introduce some notation. Let
$\mathcal{M}_3=(M,\sqsubseteq_3)$ be the intersection of the two given posets.
For $t=1,2,3$, and $m_i\in M$, let $A_t(i)$ denote the ancestral set in the
poset $(M,\sqsubseteq_t)$ corresponding to $m_i\in M$: thus
$A_t(i)=\{m_j:m_i\sqsubset_tm_j\}$. Let
$\Omega^{t,i}$ be the product of the sets $\Omega_j$ for $j\in A_t(i)$.

We have permutation groups $G(i)$, acting on sets $\Omega_i$, associated with
the points $m_i\in M$. Our products will act on the set $\Omega$, the
Cartesian product of the sets $\Omega_i$ for $m_i\in M$.

As we have seen, the generalised wreath product over $\mathcal{M}_i$ is a
product of components, where the $i$th component $F_t(i)$ consists of all
functions from $\Omega^{t,i}$ to $G_i$. Since these functions have different
domains, we cannot directly compare them. So we extend the functions in
$F_t(i)$ so that their domain is the whole of $\Omega$, with the proviso
that they do not depend on coordinates outside $\Omega^{t,i}$.

Now we have
\[F_1(i)\cap F_2(i)=F_3(i).\]
For functions in this intersection do not depend on coordinates outside
$\Omega^{1,i}$ or on coordinates outside $\Omega^{2,i}$, and so do not depend
on coordinates outside $\Omega^{1,i}\cap\Omega^{2,i}$. But, from the definition
of the intersection of posets, we have
\[\Omega^{1,i}\cap\Omega^{2,i}=\Omega^{3,i},\]
so $F_1(i)\cap F_2(i)$ is identified with the set of functions from
$\Omega^{3,i}$ to $G(i)$, and the result follows.

Taking the product over all $i$ shows (a).
  
\medskip

(b) If $(M,\sqsubseteq_1)$ is included in $(M,\sqsubseteq_2)$, then
the intersection of these two posets is just the first, and so the
same relation holds for the generalised wreath products, whence the first
is a subgroup of the second. \qed

\medskip

A \emph{linear extension} of a poset $M$ is a total order which includes the
poset. It is a standard result that a poset is the intersection of all its
linear extensions. (If $i$ is below $j$ in the poset, then $i$ is below $j$ in
every linear extension. Conversely, if $i$ and $j$ are incomparable, there is a linear extension in which $i$ is below $j$, and one in which
$j$ is below $i$.)

If $G_i$ is a permutation group on $\Omega_i$ for $i=1,2,\ldots,N$, then the
\emph{iterated wreath product} of these groups is
\[(\cdots(G_1\wr G_2)\wr\cdots\wr G_N).\]
Thus, it is the generalised wreath product of the groups over the standard
linear order on $\{1,2,\ldots,N\}$. (In fact the brackets are not necessary
since the wreath product is associative.)

\begin{cor}
A generalised wreath product of a family of groups over a poset
$(M,\sqsubseteq)$ is equal to the intersection of the iterated wreath products
over all the linear extensions of $(M,\sqsubseteq)$.
\label{c:linext}
\end{cor}

This is immediate from Theorem~\ref{t:poset2gwp} and the comments before the
corollary.

\section{Miscellanea}
\label{s:misc}

\subsection{Computing questions}
\label{s:comput}

As we did for pre-primitivity in \cite{a-mcs}, it would be good to go through
the list of small transitive groups to see how many have the OB property.
Here are some thoughts.

A permutation group $G$ on $\Omega$ is \emph{$2$-closed} if every permutation
which preserves every $G$-orbit on $2$-sets belongs to $G$. The
\emph{$2$-closure} is the smallest $2$-closed group containing $G$, and
consists of all permutations of $\Omega$ which preserve all $G$-orbits on
$\Omega^2$.

\begin{prop}
A transitive permutation group has the OB property if and only if its
$2$-closure does.
\end{prop}

For the group and its $2$-closure preserve the same binary relations, and in
particular the same equivalence relations.

So we can simplify the computation by first filtering out the $2$-closed 
groups and testing these. The computer algebra system \textsf{GAP}~\cite{gap}
has a \texttt{TwoClosure} function.

Also, \textsf{GAP} has a function \texttt{AllBlocks}. Using this we can compute
representatives of the blocks of imprimitivity and test the permuting property.
We find, for example, that only one of the transitive groups of degree~$8$ (the
dihedral group acting regularly) fails the OB property.

Table~\ref{t:nos} is a table corresponding to the one in~\cite{a-mcs}. This
gives the numbers of transitive groups of degree $n$ and the numbers with the
OB and PP properties (where PP is pre-primitivity). In the cases where OB holds
we should determine which ones give rise to isomorphic orthogonal block
structures.

\begin{table}[htbp]
\[\begin{array}{|c|rrr|}
\hline
n & \hbox{Trans} & \hbox{OB} & \hbox{PP}\\
\hline
10 & 45 & 44 & 42 \\
11 & 8 & 8 & 8 \\
12 & 301 & 285 & 276 \\
13 & 9 & 9 & 9 \\
14 & 63 & 62 & 59 \\
15 & 104 & 104 & 102 \\
16 & 1954 & 1886 & 1833 \\
17 & 10 & 10 & 10 \\
18 & 983 & 922 & 900 \\
19 & 8 & 8 & 8 \\
20 & 1117 & 1100 & 1019 \\
\hline
\end{array}\]
\caption{\label{t:nos}Numbers of transitive, OB, and pre-primitive groups}
\end{table}

Here is another approach. Taking both approaches would be a useful check on
the correctness of the computations. This uses the fact that an orthogonal
block structure gives rise to an association scheme. 

Hanaki and Miyamoto~\cite{hm} have a web page listing the association schemes
on small numbers of points. (By ``association scheme'' they mean a homogeneous
coherent configuration, which is more general than the definition in 
\cite{bailey:as}.) Now we should check which association schemes come from
orthogonal block structures, and which of these have transitive automorphism
groups.

In fact, there is a \textsf{GAP} package~\cite{as_gap} by Bamberg, Hanaki and Lansdown
which can be used to check isomorphism. Using this package, we could add a
column to the above table giving the number of different association schemes
which result (and identifying them in the Hanaki--Miyamoto tables).

\subsection{Some problems}

\qquad 1. Under what conditions is the generalised wreath product of OB groups OB?

2. In \cite{diag}, the \emph{diagonal group} $D(G,n)$ is defined for any group
$G$ and positive integer $n$, and the conditions for this group to be primitive
are determined. For which $G$ and $n$ does $D(G,n)$ have the OB property?
the PB property?

\medskip

3. In \cite{a-mcs}, the set of natural numbers $n$ for which every transitive
group of degree~$n$ is pre-primitive was considered. We can ask the analogous
question for the OB property. As we saw, there are examples of products of two
primes which are in the second set but not the first, such as $15$.

\paragraph{Conjecture} If $p$ and $q$ are primes with $p>q$ and $q\nmid p-1$,
then every transitive group of degree $pq$ has the OB property.

As well as $15$, this is true for degrees $33$ and $35$.

\medskip

4. In \cite{bailey:as} it is explained how, given an orthogonal block structure
on $\Omega$, the vector space $\mathbb{R}^\Omega$ can be decomposed into
pairwise orthogonal subspaces (called \emph{strata} in the statistical 
literature). If the group $G$ has the OB property, it preserves the subspaces
in this decomposition. When does it happen that some or all of the subspaces
are irreducible as $G$-modules?

More generally, what information does the permutation character give about
groups with the OB property?

\medskip

5. A topic worth considering is the extensions of the groups considered in this
paper by groups of lattice automorphisms, as suggested at the end of
Section~\ref{s:unify}.

\medskip

6. It would be interesting to know more about transitive groups which do not
have the OB property. How common are they? Are similar techniques useful in
their study?

\end{document}